\documentclass[10pt]{amsart}
\usepackage{eucal, mathrsfs}
\usepackage{amsmath,amssymb,amsthm,amsfonts,amscd,amsopn}
\usepackage[OT2,OT1,T1]{fontenc}
\usepackage{newcent}
\usepackage{xspace}
\usepackage{epsfig,epic,eepic,latexsym,color}
\usepackage[all]{xy}
\usepackage{comment} % I ADDED THIS TO MAKE MULTILINE COMMENTS
\usepackage{fullpage}

\DeclareSymbolFont{cyrletters}{OT2}{wncyr}{m}{n}
\DeclareMathSymbol{\Sha}{\mathalpha}{cyrletters}{"58}

\usepackage{latexsym}
\newcommand{\field}[1]{\mathbf #1}
\newcommand{\mf}[1]{\mathfrak #1}
\newcommand{\mc}[1]{\mathcal #1}
\newcommand{\ms}[1]{\mathscr #1}
\newcommand{\widebar}[1]{\overline{#1}}

\newcommand{\R}{\field R}

\newcommand{\F}{\field F}
\newcommand{\Z}{\field Z}
\newcommand{\Q}{\field Q}

\newcommand{\simto}{\stackrel{\sim}{\to}}
\newcommand{\eps}{\varepsilon}
\renewcommand{\phi}{\varphi}

\newcommand{\Hom}{\operatorname{Hom}}
\newcommand{\shom}{\ms H\!om}

\newcommand{\rshom}{\mathbf{R}\shom}

\DeclareMathOperator{\sh}{Sh}

\newcommand{\send}{\ms E\!nd}

\newcommand{\spec}{\operatorname{Spec}}

\renewcommand{\P}{\field P}

\DeclareMathOperator{\Pic}{Pic}
\newcommand{\sPic}{\ms P\!ic}

\DeclareMathOperator{\pr}{pr}

\newcommand{\m}{\boldsymbol{\mu}}

\newcommand{\G}{\field G} %for the multiplicative and additive groups

\renewcommand{\H}{\operatorname{H}}

\DeclareMathOperator{\ext}{\operatorname{Ext}}

\newcommand{\ch}{\operatorname{char}}
\DeclareMathOperator{\per}{per}
\DeclareMathOperator{\ind}{ind}

\DeclareMathOperator*{\tensor}{\otimes}

\DeclareMathOperator{\rk}{\operatorname{rk}}

\newcommand{\inj}{\hookrightarrow}

\DeclareMathOperator{\aut}{\operatorname{Aut}}
\DeclareMathOperator{\isom}{\operatorname{Isom}}
\DeclareMathOperator{\M}{\operatorname{M}}
\DeclareMathOperator{\Br}{\operatorname{Br}}
\newcommand{\invlim}{\varprojlim}

\DeclareMathOperator{\B}{\operatorname{\sf B\!}}

\newtheorem{lem}{Lemma}[subsection]

\makeatletter
\@addtoreset{lem}{section}
\@addtoreset{lem}{subsection}
\@addtoreset{lem}{subsubsection}
\@addtoreset{lem}{paragraph}
\makeatother

\renewcommand{\thelem}{\ifnum\value{subsubsection}>0{\thesubsubsection.\arabic{lem}}\else{\ifnum\value{subsection}>0{\thesubsection.\arabic{lem}}\else{\thesection.\arabic{lem}}\fi}\fi}

\newtheorem{thm}[lem]{Theorem}
\newtheorem*{theorem}{Theorem}

\newtheorem{conj}[lem]{Conjecture}
\newtheorem*{standard conjecture}{Standard Conjecture}
\newtheorem{prop}[lem]{Proposition}

\newtheorem{cor}[lem]{Corollary}

\newtheorem{claim}[lem]{Claim}

\newtheorem*{substandard conjecture}{Standard Conjecture}

\theoremstyle{definition}
\newtheorem{defn}[lem]{Definition}

\newtheorem{example}[lem]{Example}

\newtheorem{hyp}[lem]{Hypothesis}

\theoremstyle{remark}

\newtheorem{rem}[lem]{Remark}

\newtheorem{notn}[lem]{Notation}
\newtheorem{ques}[lem]{Question}

\theoremstyle{definition}

\theoremstyle{remark}

\author{Max Lieblich}
\address{Fine Hall, Washington Road, Princeton, NJ}
\thanks{The author was partially supported by NSF grant DMS-0758391. He would like to thank the Clay Mathematics Institute for hosting the conference at which some of this work was presented and Peter Newstead for his many inspiring papers}

\theoremstyle{remark}

\title {Arithmetic aspects of moduli spaces of sheaves on curves}

\begin{document}
\begin{abstract}
  We describe recent work on the arithmetic properties of moduli
  spaces of stable vector bundles and stable parabolic bundles on a
  curve over a global field.  In particular, we describe a connection
  between the period-index problem for Brauer classes over the
  function field of the curve and the Hasse principle for rational
  points on \'etale forms of such moduli spaces, refining classical
  results of Artin and Tate.
\end{abstract}
\maketitle

\tableofcontents

\section* {Introduction}

The goal of this paper is to introduce the reader to recent work on
some basic arithmetic questions about moduli spaces of vector bundles
on curves. In particular, we will focus on the correspondence between
rational-point problems on (\'etale) forms of such moduli spaces and
classical problems on the Brauer groups of function fields. This might
be thought of as a non-abelian refinement of the following result of Artin
and Tate (which we state in a special case).  

Fix a global field $K$ with scheme of integers $S$. (If $\ch K > 0$,
we assume that $S$ is proper so that it is unique.) Let $f: X \to S$
be a proper flat morphism of relative dimension 1 with a section $D
\subset X$ and smooth generic fiber $X_\eta$.  Let $\Br_\infty(X)$
denote the kernel of the restriction map
$\Br(X)\to\prod_{\nu|\infty}\Br(X\tensor K_\nu)$, i.e., the Brauer
classes which are trivial at the fibers over the Archimedean places.
(If $\ch K>0$ or $K$ is totally imaginary, $\Br_\infty(X)=\Br(X)$.)

\begin{theorem}[Artin-Tate]
  There is a natural isomorphism $$\Br_\infty(X)\simto\Sha^1(\spec
  K,\operatorname{Jac}(X_\eta)).$$
\end{theorem}

From a modern point of view, this isomorphism arises by sending a
$\m_n$-gerbe $\ms X\to X$ to the relative moduli space of invertible
twisted sheaves of degree $0$.  (We have provided a brief review of
gerbes, moduli of twisted sheaves, and the Brauer group in the form of
an appendix.) Our goal will be to study the properties of higher-rank
moduli of twisted sheaves.  Each moduli space carries an adelic point
(is in the analogue of the Tate-Shafarevich group), but most of the
moduli spaces are geometrically rational (or at least rationally
connected) with trivial Brauer-Manin obstruction, so that the Hasse
principle is conjectured to hold.  This yields connections between
information about the complexity of Brauer classes on arithmetic
surfaces and conjectures on the Hasse principle for geometrically
rational varieties.

As the majority of work on moduli spaces of vector bundles on curves
is done in a geometric context, we give in Section \ref{Torelli
  section} a somewhat unconventional introduction to the study of
forms of moduli. This is done primarily to fix notation and introduce
some basic constructions. We also use this section to introduce a
central idea: forms of the moduli problem naturally give forms of the
stack and not merely of the coarse moduli space, and the relation
between these forms captures cohomological information which
ultimately is crucial for making the arithmetic connections. In this
geometric section, we use this philosophy to prove a silly
``non-abelian Torelli theorem''.

Starting with Section \ref{questions} we turn our attention to
arithmetic problems. In particular, after asking some basic questions
in Section \ref{questions}, we link a standard conjecture on the Hasse
principle for 0-cycles of degree 1 (Conjecture 1.5(a) of \cite{ct95}) to
recent conjectures on the period-index problem in Section
\ref{classical}. We also recast some well-known results due to Lang
and de Jong on the Brauer group in terms of rational points on moduli
spaces of twisted sheaves. This serves to create a tight link between
the Hasse principle and the period-index problem for unramified Brauer
classes on arithmetic surfaces. In particular, we get the following
refinement of the Artin-Tate isomorphism.

Let $\ms{X} \to X$ be a $\m _ n$-gerbe which is trivial on geometric
fibers of $f$ and over infinite places, and let $\alpha \in \Br_\infty
(X)$ denote its associated Brauer class.  As explained in Proposition
\ref{forms}, the stack $\ms{M} _ {\ms{X} / S} (n, \ms{O} (D))$ of flat
families of stable $\ms{X}$-twisted sheaves of rank $n$ and
determinant $\ms{O} (D)$ is a form of the stack $\ms{M} _ {X / S} (n,
\ms{O} (D))$ of (non-twisted) stable sheaves with the same discrete
invariants. We write $M _ {\ms{X} / S} (n, \ms{O} (D))$ for the coarse
moduli space of $\ms{M} _ {\ms{X} / S} (n, \ms{O} (D))$ and $M _
{\ms{X} _ K / K} (n, \ms{O} (D _ K))$ for its generic fiber over $S$.

\begin{theorem}[L]
  The class $\alpha$ satisfies $\per (\alpha) = \ind (\alpha)$ if and
  only if the Hasse principle for 0-cycles of degree 1 holds for the smooth
  projective geometrically rational $K$-scheme $M _ {\ms{X} _ K / K}
  (n, \ms{O} (D _ K))$.
\end{theorem}

Since $\Pic (M _ {\ms{X} _ {\widebar K} / {\widebar K}} (n, \ms{O} (D
_ {\widebar K}))) = \Z$, the condition of the theorem is equivalent to
the statement that the Brauer-Manin obstruction is the only
obstruction to the Hasse principle for this particular variety.  Just
as the element of $\Sha^1(\spec K,X_\eta)$ measures non-triviality of
$\alpha$, here the sucess (or failure) of the Hasse principle for
$M_{\ms X_K/K}(n,\ms O(D))$ measures the complexity of the division
algebra over $K(X)$ with Brauer class $\alpha$. As an example for
the reader unfamiliar with the period and index (reviewed in the
appendix), this says that if the Brauer-Manin obstruction
is the only one for $0$-cycles of degree $1$ on geometrically rational
varieties then any class of order $2$ in $\Br_\infty(X)$ must be the class
associated to a generalized quaternion algebra $(a,b)$ over $K(X)$.
 
When $\ch K = p >
0$, one can prove that any $\alpha \in \Br (X)$ whose period is 
relatively prime to $p$ satisfies $\per (\alpha) = \ind (\alpha)$ (see \cite{lieblich_twisted_2008}). The
proof uses the geometry of moduli spaces of twisted sheaves on
surfaces over finite fields, and therefore contributes nothing to our
understanding of the mixed-characteristic situation.

In Sections \ref{parabolic bundles} and \ref{parabolic forms} we start
to consider how one might generalize the Theorem to the case of
ramified Brauer classes.  Things here are significantly more
complicated -- for example, we know that there are ramified classes
whose period and index are unequal.  A nice collection of such
examples is provided in \cite{kunyavskiu_bicyclic_2006} in the form of
biquaternion algebras over $\Q(t)$ with Faddeev index $4$ such that
for all places $\nu$ of $\Q$ the restriction to $\Q_\nu(t)$ has index
$2$.  Can we view such algebras as giving a violation of some sort of
Hasse principle?  As we discuss in Section \ref{parabolic forms},
there is a canonical rational-point problem associated to such an 
algebra, but it fails to produce a counterexample to the Hasse
principle because it lacks a local point at some place!  In other
words, while these algebras seem to violate some sort of ``Hasse principle,''
the canonically associated moduli problems actually have \emph{local}
obstructions.  This phenomenon is intriguing and demands further investigation.

\section* {Acknowledgments}

The author thanks Jean-Louis Colliot-Th\'el\`ene for his extremely
valuable comments on the history surrounding the Brauer-Manin
obstruction and the conjectures growing out of it.  He also thanks the
referee for helpful corrections.

\section* {Notation and assumptions}

As in the introduction, given an arithmetic surface $X\to S$ with $S$
the scheme of integers of $K$, we will write $\Br_\infty(X)$ for the
subgroup of $\Br(X)$ of classes whose restriction to $X\tensor K_\nu$
is trivial for all Archimedean places of $K$.

We assume familiarity with the theory of algebraic stacks, as
explained in \cite{MR1771927}.  Given a stack $\ms X$, we will denote the inertia stack by $$\ms I(\ms X):=\ms X\times_{\Delta,\ms
  X\times\ms X,\Delta}\ms X;$$ this is the fiber product of the
diagonal with itself, and represents the sheaf of groups which assigns
to any object its automorphisms.  The notation related to moduli is
explained in the Appendix.  While gerbes, twisted sheaves, and the
Brauer group are briefly reviewed in the Appendix, the reader
completely unfamiliar with them will probably benefit from consulting
\cite{MR0344253,lieblich_moduli_2007,lieblich_twisted_2008}

Most of the material described here is derived from the more extensive
treatments in \cite{lieblich_period_2007,lieblich_twisted_2008}.  We
have thus felt free to merely sketch proofs.  The material of Sections
\ref{parabolic bundles} and \ref{parabolic forms} is new, but is not
in its final form.  Again, we have not spelled out all of the details;
they will appear once a more satisfactory understanding of the
phenomena described there has been reached.

\section {Forms of moduli and an isotrivial Torelli theorem} \label{Torelli section}

Let $C/K$ be a curve over a field with a $K$-rational point $p$ and $L$ an invertible sheaf on $C$.  The central objects of
study in this paper will be forms of the moduli stack $\ms
M_{C/K}(r,L)$ of stable vector bundles on $C$ of rank $r$ and
determinant $L$.  In this section we give a geometric introduction to
the existence and study of these forms.  It turns out that the most
interesting forms of $\ms M_{C/K}(r,L)$ are those which arise from
$\m_r$-gerbes $\ms C\to C$.  The reader unfamiliar with gerbes should
think of these as stacky forms of $C$; the curve $C$ itself
corresponds to the trivial gerbe $\B\m_r\times C$.  As a warm-up for
the rest of the paper, we give an amusing Torelli-type theorem for
such stacky forms of $C$.

We begin by discussing forms of the coarse space $M_C(r,L)$.  From a
geometric point of view, forms of a variety are simply isotrivial
families.

\begin{defn} \label{isotrivial} Given an algebraically closed field
  $k$, a morphism of $k$-schemes $f: X \to Y$ is \emph {isotrivial} if
  there is a faithfully flat morphism $g: Z \to Y$, a $k$-scheme $W$,
  and isomorphism $X \times _ Y Z \simto W \times Z$.
\end{defn}

Isotrivial families arise naturally from geometry, as the following
example shows.

\begin{example} \label{determinant example}

  Suppose $C _ \epsilon$ is an infinitesimal deformation of $C$ over
  $k [\epsilon]$ and $L _ \epsilon$ and $L ' _ \epsilon$ are two
  infinitesimal deformations of $L$. We claim that the moduli spaces
  $M _ {C _ \epsilon / k [\epsilon]} (r, L _ \epsilon)$ and $M _ {C _
    \epsilon / k [\epsilon]} (r, L ' _ \epsilon)$ are isomorphic.  To
  see this, note that $L ' _ \epsilon \tensor L ^{-1} _ \epsilon$ is
  an infinitesimal deformation of $\ms{O} _ C$, is thus an $r$th power
  in $\Pic (C _ \epsilon)$. Twisting by an $r$th root gives an
  isomorphism of moduli problems.  This obviously generalizes to
  deformations of $C$ over Artinian local $k$-algebras $A$.  As a
  consequence, if $R$ is a complete local $k$-algebra with maximal ideal
  $\mf m$ and residue field $k$ and $\mc L$ is an invertible sheaf of
  degree $d$ on $C\tensor R$ with reduction $L$ over the residue
  field, we see that $M_{C\tensor R/R}(r,\mc L)\cong
  M_{C/k}(r,L)\tensor R$.  Indeed, since $M_{C\tensor R/R}(r,\mc L)$
  and $M_{C/k}(r,L)$ are proper, the Grothendieck existence theorem
  shows that the natural map
\begin{align*}
\isom_R(M_{C\tensor R/R}(r,\mc L),M_{C/k}(r,L)\tensor R)&\to\\
\invlim\isom_{R_n}&(M_{C\tensor R_n/R_n}(r,\mc L_{R_n}),M_{C/k}(r,L)\tensor R_n)
\end{align*}
is an isomorphism, where $R_n:=R/\mf m^{n+1}$.  Since the identity map
over the residue field lifts to each infinitesimal level, we conclude
that there is an isomorphism over $R$.

Now, given $r$ and $d$ as above, let $M _ C (r, d)$ denote the moduli
space of all stable vector bundles of rank $r$ and degree $d$. The
determinant defines a morphism $\det: M _ C (r, d) \to \Pic _ {C / k}
^ d$. By the argument of the previous paragraph, all of the fibers of
$\det$ are mutually isomorphic.  This is a basic example of an
isotrivial family which arises naturally from geometry.
\end{example}

How can we classify isotrivial families? The classification is
familiar from descent theory. Let us illustrate how this works using
Example \ref{determinant example}. To simplify notation, let $M$
denote the moduli space $M _ C (r, L)$ for some fixed $L \in \Pic ^ d
_ {C / k}$, and write $X = C \times \Pic ^ d _ {C / k}$. Consider the
$\Pic ^ d _ {C / k}$-scheme $$T = \isom _ {\Pic ^ d _ {C / k}} (M (r,
d), M \times \Pic ^ d _ {C / k});$$ this is a left torsor under the
scheme of automorphisms $\aut (M)$. We now make one simplifying
assumption which will tie the geometry to algebraic constructions we
will do later.

\begin{hyp} \label{automorphism chopper} Suppose that the curve $C$
  has trivial automorphism group, and that both $r$ and $d$ are odd.
\end{hyp}
The purpose of Hypothesis \ref{automorphism chopper} is to ensure that
the automorphism group of $M _ C (r, L)$ is given entirely by the
natural image of $\Pic _ C [r]$ in $\aut (M)$ under the map sending an
invertible sheaf $N \in \Pic _ C [r]$ to the automorphism given by
twisting by $N$ (see \cite{kouvidakis_automorphism_1995}).

Given Hypothesis \ref{automorphism chopper}, the torsor $T$
corresponds to a class
$$
\widebar {\alpha} \in \H ^ 1 (\Pic ^ d _ {C / k}, \R ^ 1 {\pr _ 2} _ {*} \m _ r),
$$
where $\pr _ 2$ is the second projection of the product $C \times \Pic
^ d _ {C / k}$.  The cohomology class thus arising belongs to a
convergent in the Leray spectral sequence for $\m _ r$ with respect to
the morphism $\pr _ 2: C \times \Pic ^ d _ {C / k} \to \Pic ^ d _ {C /
  k}$. In fact, the Leray spectral sequence gives a surjection
$$
\H ^ 2 (C \times \Pic ^ d _ {C / k}, \m _ r) \to \H ^ 1 (\Pic ^ d _ {C
  / k}, \R ^ 1 {\pr _ 2} _ {*} \m _ r).
$$
(This is surjective because the edge map $\H ^ 3 (\Pic ^ d _ {C / k},
\m _ r) \to \H ^ 3 (C \times \Pic ^ d _ {C / k}, \m _ r)$ is split by any
$k$-point $p \in C$.) There is a Chern class map
$$
\Pic (C \times \Pic ^ d _ {C / k}) \to \H ^ 2 (C \times \Pic ^ d _ {C / k}, \m _ r)
$$
arising in the usual way from the Kummer sequence.

Since $C$ has a $k$-point, there is a tautological sheaf $\ms{L}$ on
$C \times \Pic ^ d _ {C / k}$.

\begin{claim} \label{class identification} The class $\widebar
  {\alpha}$ is the image of $\ms{L} \tensor L ^{-1}$ under the
  composition of the above maps, where $L$ is the invertible sheaf
  corresponding to the base point parameterizing $M$.
\end{claim}

To prove this claim, we give a geometric interpretation of the maps
\begin{equation} \label{first map} \Pic (C \times \Pic ^ d _ {C / k})
  \to \H ^ 2 (C \times \Pic ^ d _ {C / k}, \m _ r)
\end{equation}
and
\begin{equation} \label{second map} \H ^ 2 (C \times \Pic ^ d _ {C /
    k}, \m _ r) \to \H ^ 1 (\Pic ^ d _ {C / k}, \Pic _ {C / k} [r]).
\end{equation}

To interpret (\ref{first map}): given an invertible sheaf $\ms{N}$ on
$C \times \Pic ^ d / {C / k}$, we get a $\m _ r$-gerbe $[\ms{N}] ^ {1
  / r}$ over $X$. This is a stack, i.e., a moduli problem, which in
this case is very explicit. Given an $X$-scheme $T \to X$, an object
of $[\ms{N}] ^ {1 / r}$ is given by a pair $(\Lambda, \psi)$ with
$\Lambda$ an invertible sheaf on $T$ and $\psi: \Lambda ^ {\tensor r}
\to \ms{N}$ an isomorphism.

To interpret (\ref{second map}): given a $\m _ r$-gerbe $\ms{X} \to
X$, we can consider the pushforward stack $\pr _ 2 \ms{X} \to \Pic ^ d
_ {C / k}$.

\begin{claim} \label{gerbe torsor mama} With the above notation, given
  a morphism $f: X \to Y$ the stack $f _ {*} \ms{X}$ is a $f _ {*} \m
  _ r$-gerbe over a $\R ^ 1 {f} _ {*} \m _ r$-pseudotorsor. If $\ms{X}
  \times _ Y U$ is a neutral gerbe over $X \times _ Y U$ for some
  \'etale surjection $U \to Y$, sheafification of $f _ {*} \ms{X}$ is
  an \'etale torsor.
\end{claim}

\begin{proof} [Sketch of proof]
  By definition, an object of $f _ {*} \ms{X}$ over a $Y$-scheme $S
  \to Y$ is an object of $\ms{X}$ over $S \times _ Y X$. It
  immediately follows that $f _ {*} \ms{X}$ is a gerbe banded by $f _
  {*} \m _ r$. Let the sheafification of $f _ {*} \ms{X}$ be denoted
  $\Theta \to Y$. Twisting by torsors gives an action $\R ^ 1f _ {*}
  \m _ r \times \Theta \to \Theta$. One checks that this makes
  $\Theta$ into a pseudotorsor; the hypothesis of the last sentence of
  the claim transparently makes $\Theta$ have local sections in the
  \'etale topology, and therefore a torsor.
\end{proof}

Sending $\ms{X}$ to the sheafification of $\pr _ 2 \ms{X}$ gives an interpretation of (\ref{second map}).

\begin{proof} [Proof of Claim \ref{class identification}]
  Let $\ms{X} \to X$ be the $\m _ r$-gerbe $[\ms{L} \tensor L ^{-1}] ^
  {1 / r}$ defined above, and let $\Theta$ denote the sheafification
  of $\ms{X}$ (as a $\Pic ^ d _ {C / k}$-stack). Tensor product
  defines a morphism of $\Pic ^ d _ {C / k}$-stacks
$$
{\pr _ 2} _ {*} \ms{X} \times \ms{M} (r, L) _ {\Pic ^ d _ {C / k}} \to
\ms{M} (r, d),
$$
where $\ms{M}$ is used in place of $M$ to denote the stack instead of
the coarse moduli space. Passing to sheafifications yields a map
$$
\Theta \times M (r, L) \to M (r, d)
$$
which is compatible with the natural $\R ^ 1 {\pr _ 2} _ {*} \m _ r$-actions. By adjunction, this gives a $\Pic (C) [r]$-equivariant map
$$
\Theta \to \isom (M (r, L), M (r, d)),
$$
yielding the desired result.
\end{proof}

This analysis of Example \ref{determinant example} fits into a
general picture. Let $f: C \to S$ be a proper smooth relative curve of
genus $g \geq 2$. Suppose $\ms{C} \to C$ is a $\m _ r$-gerbe whose
associated cohomology class $[\ms{C}] \in \H ^ 2 (C, \m _ r)$ is
trivial on all geometric fibers of $f$.

\begin{prop} \label{forms} There is an \'etale surjection $U \to S$
  and an isomorphism of stacks $\tau: \ms{M} _ {\ms{C} _ U / U} (r, L)
  \simto \ms{M} _ {C _ U / U} (r, L)$.
\end{prop}
In other words, the stack of stable $\ms{C}$-twisted sheaves of rank
$r$ determinant $L$ is a form of the stack of stable sheaves on $C$ of
rank $r$ and determinant $L$. (As the reader will note from the proof,
it is essential that the cohomology class be trivial on geometric
fibers for this to be true.)

\begin{proof} [Sketch of proof]
  The proper and smooth base change theorems and the compatibility of
  the formation of \'etale cohomology with limits show that there is
  an \'etale surjection $U \to S$ such that $[\ms{C}] _ U = 0$. Thus,
  it suffices to show that if $[\ms{C}] = 0$ (in other words, if
  $\ms{C} \cong \B \m _ {r, C}$) then there is an isomorphism $\ms{M}
  _ {\ms{C} / S} (r, L) \simto \ms{M} _ {C / S} (r, L)$. Under this
  assumption, there is an invertible $\ms{C}$-sheaf $\chi$ such that
  $\chi ^ {\tensor r}$ is isomorphic to $\ms{O}$. Tensoring with $\chi
  ^{-1}$ gives the isomorphism in question.
\end{proof}

The isomorphism $\tau$ of Proposition \ref{forms} yields a coarse
isomorphism $\widebar {\tau}: M _ {\ms{C} _ U / U} (r, L) \simto M _
{C _ U / U} (r, L)$, so that $M _ {\ms{C} / S} (r, L)$ is a form of $M
_ {C / S} (r, L)$. Just as above, we can give the cohomology class
corresponding to this form (by descent theory): it is precisely the
image of $[\ms{C}]$ under the edge map $\epsilon: \H ^ 2 (C, \m _ r)
\to \H ^ 1 (S, \R ^ 1f _ {*} \m _ r)$ in the Leray spectral
sequence. There is one mildly interesting consequence of this
fact. Since $\epsilon$ has (in general) a kernel, we see that the
coarse moduli space $M _ {\ms{C} / S} (r, L)$ does not (in general)
characterize the ``curve'' $\ms{C}$. In other words, it appears that
there is no ``stacky Torelli theorem''. It is perhaps illuminating to
give an example of the failure.

\begin{lem}\label{coho pullback}
  If $f:X\to S$ is a proper morphism with geometrically connected
  fibers such that $\Pic_{X/S}=\Z$, then the natural pullback map
  $\H^2(S,\m_r)\to\H^2(X,\m_r)$ is injective.
\end{lem}
\begin{proof}
  The Leray spectral sequence shows that the kernel of the map is the
  image of $\H^0(S,\R^1f_\ast\m_r)=\H^1(X,\Pic_{X/S}[r])=0$.
\end{proof}

\begin{example} \label{the pullback case} If $\ms{C} = C \times
  \ms{S}$ for a $\m _ r$-gerbe $\ms{S} \to S$, then (for example, by
  the above Leray spectral sequence calculation) there is an
  isomorphism $b: M _ {\ms{C} / S} (r, L) \simto M _ {C / S} (r,
  L)$. However, the stacks $\ms{M} _ {\ms{C} / S} (r, L)$ and $\ms{M} _
  {C / S} (r, L)$ are not isomorphic unless $\ms{S}$ is isomorphic to $\B \m _
  {r, S}$. In fact, viewing (via $b$) both of these stacks as $\m _
  r$-gerbes over $M _ {C / S} (r, L)$, we have an equation
$$
[\ms{M} _ {\ms{C} / S} (r, L)] - [\ms{M} _ {C / S} (r, L)] = [\ms{S} _
{M _ {C / S} (r, L)}]
$$
as described in \cite{krashen_index_2008}. By Lemma \ref{coho
  pullback}, we see that $[\ms S]=0$ if and only if $[\ms
S_{M_{C/S}(r,L)}]=0$, as desired.
\end{example}

On the other hand, if we keep track of the stacky structure, we have
the following silly ``Torelli'' theorem.  Let $f: C \to S$ be a proper
smooth relative curve of genus $g \geq 2$ with a section whose
geometric generic fiber has no nontrivial automorphisms. Suppose $L$
is an invertible sheaf on $C$ of degree $d$ on each geometric fiber of
$f$ and $r$ is a positive integer relatively prime to $d$ such that
$rd$ is odd.  Given stacks $\ms X$ and $\ms Y$ with inclusions
$\m_r\inj\ms I(\ms X)$ and $\m_r\inj\ms I(\ms Y)$, the notation $\ms
X\cong_r\ms Y$ will mean that there is an isomorphism $\iota:\ms
X\simto\ms Y$ such that the composition
$\iota^\ast\m_r\simto\m_r\to\ms I(\ms X)\to\iota^\ast\ms I(\ms Y)$ is
the pullback under $\iota$ of the given inclusion $\m_r\to\ms I(\ms
Y)$.  We will call such an isomorphism ``$\m_r$-linear''.

\begin{thm} [Isotrivial Torelli] \label{isotrivial Torelli}

  With the above notation, if $\ms{C} _ 1$ and $\ms{C} _ 2$ are $\m _
  r$-gerbes on $C$ whose restrictions to geometric fibers of $f$ are
  trivial, then $\ms{C} _ 1 \cong _ r \ms{C} _ 2$ if and only if
  $\ms{M} _ {\ms{C} _ 1 / S} (r, L) \cong _ r \ms{M} _ {\ms{C} _ 2 /
    S} (r, L)$.
\end{thm}
\begin{proof}

  The assumption that $f$ has a section leads (via pullback and the
  relative cohomology class of the section) to a natural splitting
$$
\H ^ 2 (C, \m _ r) = \H ^ 2 (S, \m _ r) \oplus \H ^ 1 (S, \R ^ 1f _
{*} \m _ r) \oplus \H ^ 0 (S, \R ^ 2f _ {*} \m _ r)
$$
such that the first two summands correspond to classes which are
trivial on geometric fibers of $f$. As discussed above, the image of
$[\ms{C} _ i]$ in $\H ^ 1 (S, \R ^ 1f _ {*} \m _ r)$ is the class
associated to $M _ {\ms{C} _ i / S} (r, L)$. If $\ms{M} _ {\ms{C} _ 1
  / S} (r, L)$ is isomorphic to $\ms{M} _ {\ms{C} _ 2 / S} (r, L)$
(say, by an isomorphism $\phi$) then certainly the same is true for
the coarse moduli spaces (isomorphic via $\widebar {\phi}$), from
which we conclude that $[\ms{C} _ 1]$ and $[\ms{C} _ 2]$ have the same
image in $\H ^ 1 (S, \R ^ 1f _ {*} \m _ r)$. Thus, there exists some
$\m _ r$-gerbe $\ms{S} \to S$ such that $[\ms{C} _ 1] - [\ms{C} _ 2] =
[\ms{S}] |_ C$ in $\H ^ 2 (C, \m _ r)$. Using Giraud's theory \cite[\S IV.2.4]{MR0344253}, it
follows that we can write $\ms{C} _ 1 = \ms{C} _ 2 \wedge \ms{S} _
C$. In this situation, there is a canonical isomorphism $b: M _
{\ms{C} _ 1 / S} (r, L) \simto M _ {\ms{C} _ 2 / S} (r, L)$ with the
property that $[\ms{M} _ {\ms{C} _ 1 / S} (r, L)] - b ^ {*} [\ms{M} _
{\ms{C} _ 2 / S} (r, L)] = [\ms{S}] _ {M _ {\ms{C} _ 1 / S} (r, L)}$
in $\H ^ 2 (M _ {\ms{C} _ 1 / S} (r, L), \m _ r)$.

By Hypothesis \ref{automorphism chopper}, any automorphism of $M _
{\ms{C} _ 1 / S} (r, L)$ lifts to a ($\m_r$-linear) automorphism of
$\ms{M} _ {\ms{C} _ 1 / S} (r, L)$. Applying this to $b \circ \widebar
{\phi} ^{-1}$, we see that $b$ lifts to an isomorphism $\ms{M} _
{\ms{C} _ 1 / S} (r, L) \simto \ms{M} _ {\ms{C} _ 2 / S} (r, L)$. We
thus conclude that $\ms{S} _ {M _ {\ms{C} _ 1 / S} (r, L)}$ is a
trivial gerbe.  By Lemma \ref{coho pullback}, we see that $[\ms S]=0$,
so $\ms C_1$ and $\ms C_2$ are isomorphic $\m_r$-gerbes.
\end{proof}

\section {Some arithmetic questions about Brauer groups and rational points on varieties over global fields} \label{questions}

Let $C/k$ be a curve over a field.  In this section we describe how
the forms arising in the preceding section are related to basic questions
about the arithmetic of function fields.  The linkage is provided by
an interpretation of the group $\H^2(C,\m_r)$.  The inclusion $\m _ r
\to \G_m$ yields classes in $\H ^ 2 (C, \G_m)$, which is equal to the
Brauer group of $C$. The reader is referred to the appendix for a
review of basic facts about the Brauer group of a scheme.

Given a field $K$, there are several questions about the arithmetic
properties of $K$ in which the Brauer group plays a central role.
\begin{enumerate}
\item The period-index problem: given $\alpha \in \Br (K)$, what is
  the minimal $g$ such that $\ind (\alpha) \mid \per (\alpha) ^ g$?
\item The index-reduction problem: given a field extension $L / K$ and
  a class $\alpha \in \Br (K)$, how can we characterize the number
  $\ind (\alpha) / \ind (\alpha |_ L)$?
\item The Brauer-Manin obstruction to the Hasse principle: is the
  Brauer-Manin obstruction the only obstruction to the existence of
  0-cycles of degree 1?
\end{enumerate}
In this paper we will focus on questions (1) and (3).  Question (2) also has close ties to the geometry of moduli spaces; we refer the reader to \cite{krashen_index_2008} for details.  

As the third question is the most technical, let us briefly review
what it means.  Suppose $K$ is a global field with adele ring $A$ and
$X$ is a proper geometrically connected $K$-scheme. For example, $X$
could be a smooth quadric hypersurface. For smooth quadric hypersurfaces, a classical theorem of Hasse
and Minkowski says that $X (K) \ne \emptyset$ if and only if $X (K _
{\nu}) \ne \emptyset$ for all places $\nu$ of $K$ (including the
infinite ones). This principle is usually referred to as the ``Hasse
principle''. A natural question which arises from this theorem is
whether or not this principle holds for an arbitrary variety. This
turns out not to the case \cite{MR0041871}, but there is often an
explanation for the failure of this principle arising from a
cohomological obstruction discovered by Manin \cite{MR0427322}. To
describe this obstruction, few minor technical remarks are in order.

Since $A$ is a $K$-algebra, we can consider the adelic points $X
(A)$. Restriction gives a map $X (A) \to \prod _ {\nu} X (K _
{\nu})$. In fact, this map is a bijection. (To prove this, one can use
a regular proper model of $X$ over the scheme of integers of $K$ to
reduce to the case in which $X$ is affine, where this follows from the
universal property of the product.) From this point of view, the Hasse
principle says that $X (A) \ne \emptyset$ if and only if $X (K) \ne
\emptyset$. Moreover, the $K$-algebra structure on $A$ gives a map $X
(K) \to X (A)$. Manin's idea is to produce a pairing whose ``kernel''
contains the image of $X (K)$ in $X (A)$. The pairing arises as
follows: restriction gives a map
$$
\Br (X) \times X (A) \to \Br (A) \to \Q / \Z,
$$
where the last map comes from the usual local invariants of class
field theory. The standard reciprocity law implies that the Brauer
group of $K$ is in the left kernel of this pairing, yielding an
invariant map
$$
 X (A) \to \Hom(\Br(X)/\Br(K),\Q / \Z).
$$
Write $X(A)^{\Br(X)}$ for the ``kernel'' of this map (i.e., the
elements sent to the map $0:\Br(X)/\Br(K)\to\Q/\Z$).  The same
reciprocity law shows that $X (K)$ is contained in $X(A)^{\Br(X)}$. In
particular, if $X (A) ^ {\Br (X)} = \emptyset$ then $X (K) =
\emptyset$.  The obvious question concerning this pairing is the
following.

\begin{ques}\label{ugh}
If $X (A) ^ {\Br (X)} \neq \emptyset$ then is $K(K)\neq\emptyset$?
\end{ques}

As suspected from the beginning, the answer turns out to be ``no,''
but there was no counterexample until Skorobogatov discovered a
bielliptic surface with no rational points and vanishing Brauer-Manin
obstruction \cite{MR1666779}. There is a refinement of this
question due to Colliot-Th\'el\`ene that is still the subject of
much current research. (Cf.\ Conjecture 1.5(a) of \cite{ct95} and Conjecture 2.4 of \cite{ct99}.)

\begin{conj} \label{French conjecture}
If $X (A) ^ {\Br (X)} \ne \emptyset$ then there is a 0-cycle of degree 1 (over $K$) on $X$.
\end{conj}

We will call this property ``the Hasse principle for $0$-cycles''.  A
famous theorem of Saito affirms Conjecture \ref{French conjecture}
when $X$ is a curve, under the assumption that the Tate-Shafarevich
group of the curve is finite (the original is
\cite{saito_observationsmotivic_1989} with another account of this
result in \cite{ct99}); the general case is still wide open.
According to Colliot-Th\'el\`ene, it is not known if Skorobogatov's
negative answer to Question \ref{ugh} has a $0$-cycle of degree $1$.
One of our primary goals in this paper will be to link certain cases
of Conjecture \ref{French conjecture} to the period-index problem for
function fields of arithmetic surfaces.

There is one case of Conjecture \ref{French conjecture} which will
come up below.

\begin{conj} \label{French conjecture Junior}
If $X$ is smooth and geometrically rational and $\Pic (X \tensor \widebar {K})$ is isomorphic to $\Z$ then the Hasse principle holds (for $0$-cycles) for $X$.
\end{conj}
\begin{proof} [Proof that Conjecture \ref{French conjecture Junior}
  follows from Conjecture \ref{French conjecture}] By
  \cite{grothendieck_le_1968-2}, we know $\Br (X \tensor \widebar {K})
  = 0$, since $X$ is smooth and geometrically rational. The Leray spectral
  sequence for $\G_m$ then shows that $\Br (X) / \Br (K) = \H ^ 1 (K,
  \Z)$, and the latter group is trivial (since the first cohomology of
  a finite group acting trivially on $\Z$ is trivial, and the Galois
  cohomology is a colimit of such).  Thus, $X(A)=X(A)^{\Br(X)}$.
\end{proof}

The statement of Conjecture \ref{French conjecture Junior} is meant to
include both the strong form (classical Hasse principle) and weak form
(Hasse principle for $0$-cycles).  We will discuss two different
relationships between Conjecture \ref{French conjecture Junior} and
the period-index problem; one will relate to the strong form, while
one will relate to the weak form.

\section {Moduli spaces of stable twisted sheaves on curves and period-index theorems} \label{classical}

In this section, we start to explain the connections among the various
problems described in the preceding section. In particular, we will
use classical theorems about rational points on various kinds of
varieties over various $C _ 1$-fields to solve moduli problems
encoding period-index problems.  Then we will prove the Theorem from
the Introduction.

To begin, we give another result which is a transparent translation of
a simple spectral sequence argument.

\begin{thm} \label{curve over finite field has trivial Brauer group}
If $C$ is a proper curve over a finite field $\F _ q$ then $\Br (C) = 0$.
\end{thm}
\begin{proof} [Sketch of proof]
  An exercise in deformation theory (see Section \ref{defmn}) reduces
  the theorem to the case in which $C$ is smooth. Let $\ms{C} \to C$
  be a $\m _ n$-gerbe. Consider the stack $\ms{M} _ {\ms{C} / \F _ q}
  (1, 0)$ parameterizing invertible $\ms{C}$-twisted sheaves of degree
  0. Just as in the classical case, $\ms{M} _ {\ms{C} / \F _ q} (1,
  0)$ is a $\G_m$-gerbe over a $\operatorname{Jac} (C)$-torsor $T$. By
  Lang's theorem, $T$ has a rational point $p$. By Wedderburn's
  theorem, $p$ lifts to an object of the stack, giving an invertible
  $\ms{C}$-twisted sheaf. As described in the appendix, this
  invertible twisted sheaf trivializes the Brauer class associated to
  $[\ms{C}]$. Since $\ms{C}$ was an arbitrary $\m _ n$-gerbe, this 
  shows that the Brauer group of $C$ is trivial.
\end{proof}

Next, we will sketch a proof of the following theorem.

\begin{thm} [de Jong] \label{de Jong theorem} Let $X$ be a surface
  over an algebraically closed field $k$. For all $\alpha \in \Br (k
  (X))$, we have $\per (\alpha) = \ind (\alpha)$.
\end{thm}

Unlike Theorem \ref{curve over finite field has trivial Brauer group},
this is not merely a geometric realization of a standard cohomological
argument.  There are various proof of this result -- de Jong's
original proof \cite{MR2060023}, a proof due to de Jong and Starr
\cite{thing}, and the one we present (found with details in
\cite{lieblich_twisted_2008}). They each ultimately rest on
deformation theory and the definition of a suitable moduli problem.  The
latter two both reduce the result to the existence of a section for a
rationally connected fibration over a curve and use the
Graber-Harris-Starr theorem.

\begin{proof} [Sketch of proof]
  We assume that $\ch k = 0$ for the sake of simplicity; the reader
  will find a reduction to this case in
  \cite{lieblich_twisted_2008}. We proceed in steps. We will write $n$
  for the period of $\alpha$.

  (1) Blowing up in the base locus of a very ample pencil of divisors
  (one of which contains the ramification divisor of $\alpha$) on $X$, we
  may assume that there is a fibration $X \to \P ^ 1$ with a section
  such that the ramification of $\alpha$ is entirely contained in a
  fiber and the generic fiber is smooth of genus $g\geq 2$. Let $C / k
  (t)$ be the generic fiber; this is a proper smooth curve of genus $g
  \geq 2$ with a rational point $p$, and $\alpha$ lies in $\Br (C)$.

  (2) Choose a $\m _ n$-gerbe $\ms{C} \to C$ such that $[\ms{C}] =
  \alpha \in \Br (C)$ and $[\ms{C} \tensor \widebar {k (t)}] = 0 \in
  \H ^ 2 (C, \m _ n)$. (We can do this because $C$ has a rational point.)

  (3) Consider the modular $\m _ n$-gerbe $\ms{M} _ {\ms{C} / k
    (t)} (n, \ms{O} (p)) \to M _ {\ms{C} / k (t)} (n, \ms{O} (p))$. As
  discussed in Section \ref{Torelli section}, there is an isomorphism
$$
M _ {\ms{C} / k (t)} (n, \ms{O} (p)) \tensor \widebar {k (t)} \simto M
_ {C \tensor \widebar {k (t)} / \widebar {k (t)}} (n, \ms{O} (p)),
$$
and we know that the latter (hence, the former) is unirational (in
fact, rational), and thus rationally connected.

(4) The Graber-Harris-Starr theorem implies that there is a rational
point $q \in M _ {\ms{C} / k (t)} (n, \ms{O} (p)) (k (t))$.

(5) Tsen's theorem implies that $q$ lifts to an object $\ms{V}$: a
locally free $\ms{C}$-twisted sheaf of rank $n$.

(6) The algebra $\send (\ms{V}) |_ {\eta}$ is a central division
algebra of degree $n$ with Brauer class $\alpha$, thus proving that
the index of $\alpha$ divides $n$. Since we already know the converse
divisibility relation, we are done.
\end{proof}

\begin{rem}
  The reader will note that we use the Graber-Harris-Starr theorem in
  two places in the proof: once to find a rational point on the coarse
  moduli space, and once (in the guise of Tsen's theorem) to lift that
  rational point to an object of the stack. While classical algebraic
  geometry only ``sees'' the former, the difference between fine and
  coarse moduli problems necessitates the latter.
\end{rem}

If we start with an arithmetic surface instead of a surface over an
algebraically closed field, things get more complicated.  (For
example, it is no longer true that the period and index are always
equal if the class is allowed to ramify.)  For the rest of this
section, we will discuss unramified classes.  In Section
\ref{parabolic forms} below, we will discuss certain ramified classes.
In both cases, we will tie the period-index problem to Conjecture
\ref{French conjecture Junior}.

Let $K$ be a global field, $S$ the scheme integers of $K$, and $C \to
S$ a proper relative curve with a section and smooth generic fiber. (When
$\ch (K) > 0$, the scheme of integers is assumed to be proper over the
prime field; this ensures that it is unique.)

\begin{thm} \label{Hasse principle period-index theorem} If Conjecture
  \ref{French conjecture Junior} is true, then any $\alpha \in \Br_\infty
  (C)$ satisfies $\per (\alpha) = \ind
  (\alpha)$.
\end{thm}
\begin{proof}
  Write $n=\per(\alpha)$.  Since $\mc C\to S$ has a section, we can
  choose a $\m_n$-gerbe $\ms C\to C$ such that $[\ms C\tensor\widebar
  K]=0\in\H^2(C\tensor\widebar K,\m_n)$.  By class field theory, the
  restriction of $\alpha$ to the point $p\in C(K)$ is trivial.

  Consider the stack $\chi:\ms M_{\ms C_K/K}(n,\ms O(p))\to M_{\ms
    C_K/K}(n,\ms O(p))$.  We claim that to prove the theorem it
  suffices to show that for every place $\nu$ of $K$, the category
  $\ms M_{\ms C_K/K}(n,\ms O(p))_{K_{\nu}}$ is nonempty.  Indeed, the
  map $\chi$ is a $\m_n$-gerbe; let $\beta\in\Br(M_{\ms C_K/K}(n,\ms
  O(p))$ be the associated Brauer class.  Since $M_{\ms C_K/K}(n,\ms
  O(p))$ is geometrically rational with Picard group $\Z$
  \cite{king_rationality_1999} and has a point over every completion
  of $K$, we know that the pullback map $\Br(K)\to\Br(M_{\ms
    C_K/K}(n,\ms O(p))$ is an isomorphism.  Thus, $\beta$ is the
  pullback of a class over $K$.  The fact that each $\ms M_{\ms
    C_K/K}(n,\ms O(p))_{K_\nu}$ is non-empty implies that
  $\beta_{K_\nu}=0$; class field theory again shows that $\beta=0$.
  But then any rational point of $M_{\ms C_K/K}(n,\ms O(p))$ lifts to
  an object of $\ms M_{\ms C_K/K}(n,\ms O(p))$.

  Let us assume we have found a collection of local objects as in the
  previous paragraph.  The assumption that Conjecture \ref{French
    conjecture Junior} holds yields a $0$-cycle of degree $1$ on the
  coarse space $M_{\ms C_K/K}(n,\ms O(p))$, which lifts to the stack,
  producing a complex $P^\bullet$ of locally free $\ms C_K$-twisted
  sheaves such that $\rk P^\bullet=n$.  Indeed, if there is a $\ms
  C\tensor L$-twisted sheaf of rank $n$ for some finite algebra $L/K$
  then pushing forward along $\ms C\tensor L\to\ms C$ gives a $\ms
  C\tensor K$-twisted sheaf of rank $[L:K]n$.  A $0$-cycle of degree
  $1$ yields two algebras $L_1/K$ and $L_2/K$ such that
  $[L_1:K]-[L_2:K]=1$ and $\ms M_{\ms C\tensor L_i/L_i}(n,\ms
  O(p))_{L_i}\neq\emptyset$.  There result two $\ms C\tensor
  K$-twisted sheaves $V_1$ and $V_2$ such that $\rk V_1-\rk V_2=n$,
  whence we can let $P^\bullet$ be the complex with $V_1$ in degree
  $0$ and $V_2$ in degree $1$ (and the trivial differential).  Since
  the category of coherent twisted sheaves over the generic point of
  $C$ is semisimple (it is just the category of finite modules
  over a division ring), it follows that there is a $\ms
  C_\eta$-twisted sheaf of rank $n$, whence
  $\per(\alpha)=\ind(\alpha)$, as desired.

  So it remains to produce local objects of $\ms M_{\ms C_K/K}(n,\ms
  O(p))_{K_{\nu}}$ for all places $\nu$.  If $\nu$ is Archimedean,
  then $[\ms C\tensor K_\nu]=0\in\H^2(C\tensor K_\nu,\m_n)$ by
  assumption, so that stable $\ms C\tensor{K_\nu}$-twisted sheaves are
  equivalent (upon twisting down by an invertible $\ms C\tensor
  K_\nu$-twisted sheaf) to stable sheaves on $C\tensor K_\nu$.  Since
  moduli spaces of stable vector bundles with fixed invariants have
  rational points over every infinite field, we find a local object.
  
  Now assume that $\nu$ is finite, and let $R$ be valuation ring of
  $K_\nu$, with finite residue field $F$.  By Theorem \ref{curve over
    finite field has trivial Brauer group} and the assumption that
  $[\ms C\tensor\widebar F]=0\in\H^2(C\tensor\widebar F,\m_n)$ we know that $[\ms
  C\tensor F]=0\in\H^2(C\tensor F,\m_n)$.  It follows as in the
  previous paragraph that it suffices to show the existence of a
  stable sheaf on $C\tensor F$ of determinant $\ms O(p)$ and rank $n$.
  Consider the stack $\ms M_{C\tensor F/F}(n,\ms O(p))\to M_{C\tensor
    F/F}(n,\ms O(p))$.  Just as above, the space $M_{C\tensor
    F/F}(n,\ms O(p))$ is a smooth projective rationally connected
  variety.  By Esnault's theorem \cite{esnault}, it has a rational point.  Since $\ms
  M_{C\tensor F/F}(n,\ms O(p))\to M_{C\tensor F/F}(n,\ms O(p))$ is a
  $\m_n$-gerbe and $F$ is finite, the moduli point lifts to an object,
  giving rise to a stable $\ms C\tensor F$-twisted sheaf $V$ of rank
  $n$ and determinant $\ms O(p))$.  Since $\ms M_{\ms C/S}(n,\ms
  O(p))$ is smooth over $S$, the sheaf $V$ deforms to a family over
  $R$, whose generic fiber gives the desired local object.
\end{proof}

\section {Parabolic bundles on $\P^1$} \label{parabolic bundles}

In this section, we review some basic elements of the moduli theory of
parabolic bundles of rank $2$ on the projective line.  We will focus
on the case of interest to us and describe it in stack-theoretic
language.  We refer the reader to \cite{borne_fibres_2007} for a
general comparison between the classical and stacky descriptions of
parabolic bundles.

Let $D=p_1+\cdots+p_r\subset\P^1$ be a reduced divisor, and let
$\pi:\ms P\to\P^1$ be the stack given by extracting square roots of
the points of $D$ as in \cite{cadman_using_2007}.  The stack has $r$
non-trivial residual gerbes $\xi_1,\dots,\xi_r$, each isomorphic to
$\B\m_2$ over its field of moduli (i.e., the residue field of $p_i$).
Recall that the category of quasi-coherent sheaves on $\B\m_2$ is
naturally equivalent to the category of representations of
$\m_2$. Given a sheaf $\ms F$ on $\B\m_2$, we will call the
representation arising by this equivalence the \emph{associated
  representation} of $\ms F$.

\begin{defn}
  A locally free sheaf $V$ on $\ms P$ is \emph{regular} if for each
  $i=1,\dots,r$, the associated representation of the restriction
  $V|_{\xi_i}$ is a direct sum of copies of the regular
  representation.
\end{defn}

\begin{defn}
  Let $\{a_i\leq b_i\}_{i=1}^r$ be elements of $\{0,1/2\}$.  A
  \emph{parabolic bundle} $V_\ast$ of rank $N$ with parabolic weights
  $\{a_i\leq b_i\}$ is a pair $(W,F)$, where $W$ is a locally free
  sheaf of rank $2$ and $F\subset W_D$ is a subbundle.  The parabolic
  degree of $V_\ast$ is
$$
\operatorname{pardeg}(V_\ast)=\deg W+\sum_i a_i(\rk W-\rk F_{p_i})+b_i(\rk F_{p_i})
$$
\end{defn}

We will only consider parabolic bundles with weights in $\{0,1/2\}$ in
this paper.  More general weights in $[0,1)$ are often useful.  The
stack-theoretic interpretation of this more general situation is
slightly more complicated; it is explained clearly in
\cite{borne_fibres_2007}.

Vistoli has defined a Chow theory for Deligne-Mumford stacks
\cite{vistoli_intersection_1989} in which pushforward defines an isomorphism
$A(\ms P)\tensor\Q\simto A(\P^1)\tensor\Q$ of Chow rings.  In
particular, any invertible sheaf $L$ on $\ms P$ has a degree, $\deg
L\in\Q$.  One can make a more ad hoc definition of the degree of an
invertible sheaf $L$ on $\ms P$ in the following way.  The sheaf
$L^{\tensor 2}$ is the pullback of a unique invertible sheaf $\ms M$
on $\P^1$, and we can define $\deg_{\ms P} L=\frac{1}{2}\deg_{\P^1}\ms
M$.  Thus, for example, $\deg\ms O(\xi_i)=[\kappa_i:k]/2$, where
$\kappa_i$ is the field of moduli of $\xi_i$ (the residue field of
$p_i$).

The following is a special case of a much more general result.  The
reader is referred to (e.g.) \cite{borne_fibres_2007} for the
generalities.

\begin{prop} \label{parabolic comparison} There is an equivalence of
  categories between locally free sheaves $V$ on $\ms P$ and parabolic
  sheaves $V_\ast$ on $\P^1$ with parabolic divisor $D$ and parabolic
  weights contained in $\{0,1/2\}$.  Moreover, we have $\deg_{\ms
    P}V=\operatorname{pardeg}(V_\ast)$.
\end{prop}
\begin{proof}
  Given $V$, define $V_\ast$ as follows: the underlying sheaf $W$ of
  $V_\ast$ is $\pi_\ast V$.  To define the subbundle $F\subset W_D$,
  consider the inclusion $V(-\sum\xi_i)\subset V$.  Pushing forward by
  $\pi$ yields a subsheaf $W'\subset W$, and we let $F$ be the image
  of the induced map $W'_D\to W_D$.

  We leave it to the reader as an amusing exercise to check that 1)
  this defines an equivalence of categories, and 2) this equivalence
  respects degrees, as claimed.
\end{proof}

\begin{defn} \label{parabolic stability}
Given a non-zero locally free sheaf $V$ on $\ms P$, the \emph{slope} of $V$ is
$$
\mu(V)=\frac{\deg V}{\rk V}.
$$
The sheaf $V$ on $\ms P$ is \emph{stable} if for all locally split
proper subsheaves $W\subset V$ we have $\mu(W)<\mu(V)$.
\end{defn}

The stability condition of Definition \ref{parabolic stability} is
identical to the classical notion for sheaves on a proper smooth
curve.  The reader familiar with the classical definition of stability
for parabolic bundles can easily check (using Proposition
\ref{parabolic comparison}) that a sheaf $V$ on $\ms P$ is stable if
and only if the associated parabolic bundle $V_\ast$ is stable in the
parabolic sense.  One can check (by the standard methods) that stable
parabolic bundles form an Artin stack of finite type over $k$.  The
stack of stable parabolic bundles of rank $n$ with fixed determinant
is a $\m_n$-gerbe over an algebraic space.

\begin{notn}
  Given $L\in\Pic(\ms P)$, let $\ms M_{\ms P/k}^\ast(n,L)$ denote the
  stack of regular locally free sheaves on $\ms P$ of rank $r$ and
  determinant $L$, and let $M^\ast_{\ms P/k}(n,L)$ denote the coarse
  space.
\end{notn}

\begin{prop}
  If $\ms M_{\ms P/k}^\ast(n,L)$ is non-empty then it is geometrically
  unirational and geometrically integral.  Moreover, the stack $\ms
  M_{\ms P/k}^\ast(2,\ms O(\sum\xi_i))$ is non-empty if $r>3$.
\end{prop}
\begin{proof}

  Basic deformation theory shows that $\ms M_{\ms P/k}^\ast(n,L)$ is
  smooth.  Thus, to show that it is integral, it suffices to show that
  it is connected.  We will do this by showing that it is unirational
  (i.e., finding a surjection onto $\ms M_{\ms P/k}^\ast(n,L)$ from an
  open subset of a projective space.  This is a standard trick using a
  space of extensions.  (There are in fact two versions, one using
  finite cokernels and one using invertible cokernels.  We show the
  reader the former, as it is useful in situations where the latter
  does not apply and the latter seems to be more easily available in
  the literature.)

  Let $V$ be a parablic sheaf of rank $n$ and determinant $L$.  Since
  $\ms M_{\ms P/k}^\ast(n,L)$ is of finite type, there exists a
  positive integer $N$ such that for every parabolic sheaf of rank $n$
  and determinant $L$, a general map $V\to W(N)$ is injective with
  cokernel $Q$ supported on a reduced divisor $E$ in $\ms
  P\setminus\{\xi_1,\dots,\xi_r\}$ such that $E\in|\ms O(nN)|$.  Let
  $U\subset|\ms O(nN)|$ be the open subset parametrizing divisors
  supported in $\ms P\setminus\{\xi_1,\dots,\xi_r\}$, and let $\ms
  E\subset\ms P\times U$ be the universal divisor.  The sheaf
  $\R^1(\pr_2)_\ast\rshom(\ms O_{\ms E},\pr_1^\ast V)$ is locally free on
  $U$, and gives rise to a geometric vector bundle $B\to U$ and an
  extension $$0\to V_{\ms P\times B}\to \ms W\to\ms Q_{\ms P\times
    B}\to 0$$ such that every extension $0\to V\to W(N)\to Q\to 0$ as
  above arises is a fiber over $B$.  Passing to the open subscheme
  $B^\circ$ over which the extension $\ms W$ is locally free with
  stable fibers yields a surjective map $B^\circ\to\ms M_{\ms
    P/k}^\ast(n,L)$ from an open subset of a projective space, proving
  that $\ms M_{\ms P/k}^\ast(n,L)$ is geometrically integral and
  unirational.

  To prove that $\ms M_{\ms P/k}^\ast(n,L)$ is nonempty is
  significantly more subtle.  The proof is similar to a result of
  Biswas \cite{biswas_criterion_1998} for parabolic bundles with
  parabolic degree $0$, but including it here would take us too far
  afield.
\end{proof}

\section {Forms of parabolic moduli via split ramification} \label{parabolic forms}

\subsection {Some generalities}

In this section we illustrate how to produce forms of the stack of
parabolic bundles on $\P^1$ from Brauer classes over $k(t)$.  For the
sake of simplicity, we restrict our attention to classes in
$\Br(k(t))[2]$ and parabolic bundles of rank $2$.  A generalization to
higher period/rank should be relatively straightforward.

Let $\alpha\in\Br(k(t))[2]$ be a Brauer class.  Suppose
$D=p_1+\dots+p_r$ is the ramification divisor of $\alpha$, and let
$\ms P\to\P^1$ be the stacky branched cover as in Section
\ref{parabolic bundles}.  By Corollary \ref{global ramification},
$\alpha$ extends to a class $\alpha'\in\Br(\ms P)[2]$.  Suppose $\ms
C\to\ms P$ is a $\m_2$-gerbe representing $\alpha'$ such that $[\ms
C\tensor\widebar k]=0\in\H^2(\ms P\tensor\widebar k,\m_2)$.  (For a
proof that $\ms C$ is itself an algebraic stack, the reader is
referred to \cite{lieblich_period_2007}.  We cannot always ensure that
the cohomology class of $[\ms C\tensor\widebar k]$ is trivial; we make
that as a simplifying assumption.  More general cases can be analyzed
by similar methods.)

\begin{defn}
  A \emph{regular $\ms C$-twisted sheaf} is a locally free $\ms
  C$-twisted sheaf $V$ such that for each $i=1,\dots,r$, the
  restriction $V_{\ms C\times_{\P^1}\spec\widebar{\kappa}_i}$ has the
  form $\ms L\tensor\rho^{\oplus m}$ for some integer $m>0$, where $\ms L$
  is an invertible $\ms C\tensor\widebar{\kappa}_i$-twisted sheaf and
  $\rho$ is the sheaf on $\B\m_2$ associated to the regular
  representation of $\m_2$.
\end{defn}
Just as in Definition \ref{parabolic stability} and Definition
\ref{stability}, we can define stable regular $\ms C$-twisted sheaves.

\begin{notn}
  Let $\ms M_{\ms C/k}^\ast(n,L)$ denote the stack of stable regular
  $\ms C$-twisted sheaves of rank $n$ and determinant $L$, and $M_{\ms
    C/k}^\ast(n,L)$ its coarse moduli space (sheafification).
\end{notn}

\begin{prop}
  For any section $\sigma$ of $\ms C\tensor k\to\ms P\tensor k$ there
  is an isomorphism $\ms M_{\ms C/k}^\ast(n,L)\tensor\widebar
  k\simto\ms M_{\ms P/k}^\ast(n,L)$.
\end{prop}
\begin{proof}
  We may assume that $k=\widebar k$.  The section $\sigma$ corresponds
  to an invertible $\ms C$-twisted sheaf $\ms L$ such that $\ms L^{\tensor
    2}=\ms O_{\ms C}$.  Twisting by $\ms L$ defines the isomorphism.
  (Note that the regularity condition implies that $n$ must be even
  for either space to be nonempty.)
\end{proof}

In other words, the stack $\ms M_{\ms C/k}^\ast(n,L)$ (resp.\ the
quasi-projective coarse moduli space $M_{\ms C/k}^\ast(n,L)$) is a
form of $\ms M_{\ms P/k}^\ast(n,L)$ (resp.\ $M_{\ms P/k}^\ast(n,L)$).

\begin{cor} \label{parabolic corollary} The space $M_{\ms
    C/k}^\ast(n,L)$ is geometrically (separably) unirational when it
  is nonempty.
\end{cor}

One can use a generalization of Corollary \ref{parabolic
  corollary} to higher genus curves and arbitrary period to give
another proof of Theorem \ref{de Jong theorem} without having to push
the ramification into a fiber, by simply taking any pencil and using
the generic points of the ramification divisor (and a point of the
base locus) to define the parabolic divisor.  (This is not
substantively different from the proof we give here.)  The main
interest for us, however, will be for arithmetic surfaces of mixed
characteristic.

\subsection {An extended example} \label{extended example}

Let $\alpha\in\Br(\Q(t))[2]$ be a class whose ramification divisor
$D\subset\P^1_\Z$ is non-empty with simple normal crossings.  Let $\ms
P\to\P^1_\Q$ be the stacky cover branched over $D$ to order $2$ as in
the first paragraph of Section \ref{parabolic bundles} above.  Let
$\ms C\to\ms P$ be a $\m_2$-gerbe with Brauer class $\alpha$; if
$[D\tensor\Q:\Q]$ is odd, one can ensure that $\ms C$ such that $[\ms
C\tensor\widebar{\Q}]\in\H^2(\ms P\tensor\widebar Q,\m_2)$ has the
form $[\Lambda]^{1/2}$ for some invertible sheaf $\Lambda\in\Pic(\ms
P\tensor\widebar{\Q})$ of half-integral degree.

\begin{defn}
  Given a field extension $L/K$ and a Brauer class $\alpha\in\Br(L)$,
  the \emph{Faddeev index of $\alpha$} is
  $\min_{\beta\in\Br(K)}\ind(\alpha+\beta_L)$.
\end{defn}

\begin{prop}\label{example point}
  The class $\alpha$ has Faddeev index $2$ if and only if the space
  $M^\ast_{\ms C/\Q}(2,L)^{ss}$ has a $\Q$-rational point for some
  invertible sheaf $L$.  If $[D\tensor\Q:\Q]$ is odd, we need only
  quantify over $L$ of half-integral degree and look for points in the
  stable locus $M^\ast_{\ms C/\Q}(2,L)$.
\end{prop}
The point of Proposition \ref{example point} is that the computation
of the Faddeev index is reduced to the existence of a rational point
on one of a sequence of geometrically rational smooth (projective if
$[D\tensor\Q:\Q]$ is odd) geometrically connected varieties over $\Q$.
\begin{proof}
  First, suppose that $P:\spec\Q\to M^\ast_{\ms C/\Q}(2,L)$ is a
  rational point.  Pulling back $\ms M_{\ms C/\Q}(2,L)\to M_{\ms
    C/\Q}(2,L)$ along $P$ yields a class $\beta=-[\ms
  S]\in\Br(\Q)[2]$.  Just as in \cite{krashen_index_2008}, we see
  that $$\ms M_{\ms C\wedge\ms S_{\ms P}/\Q}(2,L)\to M_{\ms C\wedge\ms
    S/\Q}(2,L)=M_{\ms C/\Q}(2,L)$$ is split over $P$, whence there is
  a $\ms C\wedge\ms S$-twisted sheaf of rank $2$.  This shows that
  $\alpha-\beta$ has index $2$ and thus that $\alpha$ has Faddeev
  index dividing $2$.  Since $\alpha$ is ramified, it cannot have
  Faddeev index $1$.

  Now suppose that $\alpha$ has Faddeev index $2$, so that there is
  some $\ms S\to\spec \Q$ such that there is locally free $\ms
  C\wedge\ms S$-twisted sheaf of rank $2$.  Since there is a canonical
  isomorphism $M_{\ms C\wedge\ms S/\Q}(2,L)=M_{\ms C/\Q}(2,L)$, upon
  replacing $\ms C$ by $\ms C\wedge\ms S$ we may assume that $\alpha$
  has period $2$, and is thus represented by a quaternion algebra
  $[(a,b)]\in\Br(\Q(t))$.  To prove the result, it suffices to show
  that in this case there is a regular $\ms C$-twisted sheaf of rank
  $2$.

  By Proposition \ref{lift starter}(1), it suffices to prove this for $\ms
  C\tensor\widehat{\ms O}_{\P^1,D}$.  Thus, we are reduced to the
  following: let $R$ be a complete discrete valuation ring with
  residual characteristic $0$ and $(a,b)$ a quaternion algebra over
  the fraction field $K(R)$.  Suppose $(a,b)$ is ramified, and let
  $\ms C\to\ms R_2$ be a $\m_2$-gerbe with Brauer class $[(a,b)]$.
  Then there is a regular $\ms C$-twisted sheaf of rank $2$.  To prove
  this, note that the bilinearity and skew-symmetry of the symbol
  allows us to assume that $b$ is a uniformizer for $R$ and $a$ has
  valuation at most $1$.  We will define a $\m_2$-equivariant Azumaya
  algebra $A$ in the restriction of $(a,b)$ to $R'=R[\sqrt b]$ such
  that the induced representation of $\m_2$ on the fiber $\widebar A$
  over the closed point of $R'$ is $\rho^{\oplus 2}$; it then follows
  that any twisted sheaf $V$ on $\ms R_2$ such that $\send(V)=A$ is
  regular by a simple geometric computation over the residue field.

  Let $x$ and $y$ be the standard generators for $(a,b)$, so that
  $x^2=a$ and $y^2=b$, and write $a=ub^\eps$ with $u\in R^\times$ and
  $\eps\in\{0,1\}$.  Let $\tilde x=x/\sqrt b^\eps$ and $\tilde
  y=y/\sqrt b$; we have $\tilde x^2=u$ and $\tilde y^2=1$, which means
  that $\tilde x$ and $\tilde y$ generate an Azumaya algebra with
  generic fiber $(a,b)\tensor K(R)(\sqrt b)$.  A basis for $A$ as a
  free $R'$ module is given by $1,\tilde x,\tilde y,\tilde x\tilde y$;
  this also happens to be an eigenbasis for the action of $\m_2$.  Let
  $\chi^1$ denote the non-trivial character of $\m_2$ and $\chi^0$ the
  trivial character. The eigensheaf decomposition of $A$ corresponding
  to the basis can be written as
  $\chi^0\oplus\chi^1\oplus\chi^\eps\oplus\chi^{1+\eps}$, where the
  last sum is taken modulo $2$.  For either value of $\eps$ this is
  isomorphic to $\rho^{\oplus 2}$, as desired.

  Gluing the local models (as in Proposition \ref{lift starter}(1))
  produces a regular $\ms C$-twisted sheaf $V$ of rank $2$.  Since
  $\alpha$ is non-trivial, we see that this sheaf must be stable,
  hence geometrically semistable (in fact, geometrically polystable
  \cite{krashen_index_2008}).  If $[D\tensor\Q:\Q]$ is odd, then $V$
  is semistable with coprime rank and degree, hence geometrically
  stable.
\end{proof}

Proposition \ref{example point} has an amusing consequence.  The
starting point is the following (somewhat useful) lemma. Fix a place
$\nu$ of $\Q$.

\begin{lem} \label{points everywhere} If there is an object of $\ms{M}
  ^ {*} _ {\ms{C} / \Q} (2, L) ^ {ss} _ {\Q _ {\nu}}$ then there is an
  object of $\ms{M} ^ {*} _ {\ms{C} / \Q} (2, L) _ {\Q _ {\nu}}$.
\end{lem}
\begin{proof}
  Let $V$ be a regular semistable $\ms{C} \tensor \Q _ {\nu}$-twisted
  sheaf of rank $2$ and determinant $L$, and let $Y$ be an
  algebraization of a versal deformation space of $V$. Since
  deformations of vector bundles on curves are unobstructed, we know
  that $Y$ is smooth over $\Q _ {\nu}$. On the other hand, we also
  know that the field $\Q _ {\nu}$ has the property that any smooth
  variety with a rational point has a Zariski-dense set of rational
  points. Finally, we know that the locus of geometrically stable
  $\ms{C} \tensor \Q _ {\nu}$-twisted sheaves is open and dense in $Y$. The
  result follows.
\end{proof}

In \cite{kunyavskiu_bicyclic_2006}, the authors produce examples of
biquaternion algebras $A$ over $\Q(t)$ of Faddeev index $4$ such that
for all places $\nu$ of $\Q$ the algebra $A\tensor\Q_\nu(t)$ has index
$2$.  They describe this as the failure of a sort of ``Hasse
principle''.  However, a careful examination of their examples shows
that in fact there is a local obstruction: for each class of such
algebras they write down, there is always a place $\nu$ over which
there is no regular $\ms C\tensor\Q_\nu$-twisted sheaf!  (Moreover,
the proofs that their examples work use this local failure in an
essential way, although the authors do not phrase things this way.)
Why doesn't this contradict Proposition \ref{example point}?  Because
for the place $\nu$ where things fail, one of the original
ramification sections ends up in the locus where the algebra
$A\tensor\Q_\nu$ is unramified, and now the condition that the sheaf
be regular around the stacky point over that section is non-trivial!

Let us give an explicit example. In Proposition 4.3 of
\cite{kunyavskiu_bicyclic_2006}, the authors show that the
biquaternion algebra $A = (17, t) \tensor (13, (t -1) (t -11))$ has
Faddeev index 4 while for all places $\nu$ of $\Q$ the algebra $A
\tensor \Q _ {\nu}$ has index 2. Consider $A _ {17} = A \tensor \Q _
{17}$; since 13 is a square in $\Q _ {17}$, we have that $A _ {17} =
(17, t)$ as $\Q _ {17}$-algebras. Elementary calculations show that
$(17, 1) = 0 \in \Br (\Q _ {17})$ and $(17, 11) \ne 0 \in \Br (\Q _
{17})$. It follows that any $\Q _ {17} (t)$-algebra 
Faddeev-equivalent to $A _ {17}$ has the property that precisely one
of its specializations at 1 and 11 will be nontrivial.

\begin{lem} \label{beef flank} Given a field $F$ and a nontrivial
  Brauer class $\gamma \in \Br (F) [2]$, let $\ms{G} \to \B \m _ 2
  \times \spec F$ be a $\m _ 2$-gerbe representing the pullback of
  $\gamma$ in $\Br (\B \m _ 2 \times \spec F)$. There is no regular
  $\ms{G}$-twisted sheaf of rank 2.
\end{lem}
\begin{proof}
  The inertia stack of $\ms{G}$ is isomorphic to $\m _ 2 \times \m _
  2$, where the first factor comes from the gerbe structure and the
  second factor comes from the inertia of $\B \m _ 2$. Given a
  $\ms{G}$-twisted sheaf $F$, the eigendecomposition of $F$ with
  respect to the action of the second factor gives rise to
  $\ms{G}$-twisted subsheaves of $F$. Since $\ms{G}$ has period 2, if
  $\rk F = 2$ there can be no proper twisted subsheaves.
\end{proof}

\begin{cor} \label{almond paste} In the example above, all of the sets
  $M ^ {*} _ {\ms{C} / \Q} (2, L) (\Q _ {17})$ are empty.
\end{cor}
\begin{proof}
  Suppose $Q \in M ^ {*} _ {\ms{C} / \Q} (2, L) (\Q _ {17})$. As
  above, after replacing $A _ {17}$ by a Faddeev-equivalent algebra,
  we can that $Q$ to an object of $\ms{M} ^ {*} _ {\ms{C} / \Q} (2, L)
  _ {\Q _ {17}}$. In other words, there would be a regular (stable!)
  $\ms{C} \tensor \Q _ {17}$-twisted sheaf of rank 2. But this
  contradicts Lemma \ref{beef flank} and the remarks immediately
  preceding it.
\end{proof}

Combining Proposition \ref{example point} and Corollary \ref{almond
  paste}, we see that the failure of the ``Hasse principle'' to which
the authors of \cite{kunyavskiu_bicyclic_2006} refer in relation to
the algebra $A$ is in fact a failure of the existence of a \emph
{local} point in the associated moduli problem! This is in great
contrast to the unramified case, which we saw above was directly
related to the Hasse principle. It is somewhat disappointing that the
examples we actually have of classes over arithmetic surfaces whose
period and index are distinct cannot be directly related to the Hasse
principle (except insofar as both sets under consideration are
empty!).

There is one mildly interesting question which arises out of this failure.

\begin{prop} \label{complicated} If Conjecture \ref{French conjecture
    Junior} is true then any element $\alpha \in \Br (\Q (t)) [2]$
  such that
\begin{enumerate}
\item the ramification of $\alpha$ is a simple normal crossings
  divisor $D = D _ 1+ \dots + D _ r$ in $\P ^ 1 _ \Z$ which is a union
  of fibers and an odd number of sections of $\P ^ 1 _ \Z \to \spec
  \Z$,
\item for every crossing point $p \in D _ i \cap D _ j$, both
  ramification extensions are either split, non-split, or ramified at
  $p$, and
\item all points $d$ of $D \times _ {\P ^ 1 _ \Z} \P ^ 1 _ \R$ which
  are not ramification divisors of $\alpha _ {\R (t)}$ give rise to
  the same element $(\alpha _ {\R (t)}) _ d \in \Br (\R)$
\end{enumerate}
satisfies $\per (\alpha) = \ind (\alpha)$.
\end{prop}

The second condition of the proposition is almost equivalent to the
statement that the restriction of $\alpha$ to $\Q _ {\nu} (t)$ has no
hot points (in the sense of Saltman) on $\P ^ 1 _ {\Z _ {\nu}}$; it is
not quite equivalent because Saltman's hot points are all required to
lie on intersections of ramification divisors, while some of the
ramification divisors of $\alpha$ may no longer be in the ramification
divisor of $\alpha _ {\Q _ {\nu} (t)}$. The proof of Proposition
\ref{complicated} uses Proposition \ref{lift starter} (2) as a
starting point for a deformation problem. It is similar in spirit to
the proof of Theorem \ref{Hasse principle period-index theorem} and
will be omitted.

\section {A list of questions}

We record several questions arising from the preceding discussion.
Let $C$ be a curve over a field $k$.

\begin{enumerate}
\item Are there biquaternion algebras in $\Br_\infty(\Q(t))$ of
  Faddeev index $4$ with non-hot secondary ramification (in the sense
  of Proposition \ref{complicated})?  For example, let $p$ be a prime
  congruent to $1$ modulo $3\cdot 4\cdot 7\cdot 13\cdot 17$ and
  congruent to $2$ modulo $5$.  What is the Faddeev index of the
  algebra $(p,t)\tensor(13,15(t-1)(t+13))$?
\item What is the Brauer-Manin obstruction for $M^\ast_{\ms
    P/k}(2,L)$?  If algebras as in the first question exist, is the
  resulting failure of the Hasse principle explained by the
  Brauer-Manin obstruction?
\item Let $\ms C\to C$ be a $\m_n$-gerbe.  What is the index of the
  Brauer class $\ms M_{\ms C/k}(n,\ms O)\to M_{\ms C/k}(n,\ms O)$?  If
  $C$ has a rational point, this index must divide $n$.  More
  generally, it must divide $n\ind(C)$.  Is this sharp?
\item Does every fiber of $M_{C/k}(n,L)\to\Pic^d_{C/k}$ over a
  rational point contain a rational point?  This is (indirectly)
  related to the index-reduction problem.  It is not too hard to see
  that if the rational point of $\Pic^d_{C/k}$ comes from an
  invertible sheaf then there is always a rational point in the fiber.
\item Suppose $k$ is a global field.  Let $\mc C\to S$ be a regular
  proper model of $C$.  Is there a class $\alpha\in\Br_\infty(\mc C)$ such that
  $\per(\alpha)\neq\ind(\alpha)$?  When $k$ has positive
  characteristic and the period of $\alpha$ is invertible in $k$, this
  is impossible \cite{lieblich_twisted_2008}.  The existence of such a
  class would disprove Conjecture \ref{French conjecture Junior} (even
  the strong form).

\end{enumerate}

\appendix

\section {Gerbes, twisted sheaves, and their moduli}

In this appendix, we remind the reader of the basic facts about
twisted sheaves, their moduli, and their applications to the Brauer
group. For more comprehensive references, the reader can consult 
\cite{lieblich_moduli_2007,lieblich_twisted_2008}. The basic setup
will be the following: let $\ms{X} \to X \to S$ be a $\m _ n$-gerbe on
a proper flat morphism of finite presentation. We assume $n$ is
invertible on $S$. At various points in this appendix, we will impose
conditions on the morphism.

We first recall what it means for $\ms{X} \to X$ to be a $\m _ n$-gerbe.
\begin{defn}
  A $\m _ n$-gerbe is an $S$-stack $\ms{Y}$ along with an isomorphism
  $\m _ {n, \ms{Y}} \to \ms{I} (\ms{Y})$.  We say that $\ms{Y}$ is a
  $\m _ n$-gerbe on $Y$ if there is a morphism $\ms{Y} \to Y$ such
  that the natural map $\sh (\ms{Y}) \to Y$ is an isomorphism, where
  $\sh (\ms{Y})$ denotes the sheafification of $\ms{Y}$ on the big
  \'etale site of $S$.
\end{defn}

Because $\ms{X}$ is a $\m _ n$-gerbe, any quasi-coherent sheaf
$\ms{F}$ on $\ms{X}$ admits a decomposition $\ms{F} = \ms{F} _ 0
\oplus \dots \oplus \ms{F} _ {n -1}$ into eigensheaves, where the
natural (left) action of the stabilizer on $\ms{F} _ i$ is via the
$i$th power map.

\begin{defn}
  An $\ms{X}$-twisted sheaf is a sheaf $\ms{F}$ of $\ms{O} _
  {\ms{X}}$-modules such that the natural left action (induced by
  inertia) $\m _ n \times \ms{F} \to \ms{F}$ is equal to scalar
  multiplication.
\end{defn}

\subsection {Moduli}

It is a standard fact (see, for example,
\cite{lieblich_remarksstack_2006}) that the stack of flat families of
quasi-coherent $\ms{X}$-twisted sheaves of finite presentation is an
algebraic stack locally of finite presentation over $S$. For the
purposes of this paper, we will focus on only a single case, where we
will consider stability of twisted sheaves. \emph {From now on, we
  assume that $S = \spec k$ with $k$ a field, and $X$ is a proper
  smooth curve over $k$.}

To define stability, we need a notion of degree for invertible
$\ms{X}$-twisted sheaves. For the sake of simplicity, we give an ad hoc
definition. Given an invertible $\ms{X}$-twisted sheaf $L$, we note
that $L ^ {\tensor n}$ is the pullback of a unique invertible sheaf $L
'$ on $X$. We can thus define $\deg L$ to be $\frac {1} {n} \deg L
'$. With this definition, we can define the degree of a locally free
$\ms{X}$-twisted sheaf $V$ as $\deg V = \deg \det V$. With this
notion of degree, we can define stability.

\begin{defn} \label{slope}
The \emph {slope} of a non-zero locally free $\ms{X}$-twisted sheaf $V$ is
$$
\mu (V) = \frac {\deg V} {\rk V}.
$$
\end{defn}

\begin{defn} \label{stability} A locally free $\ms{X}$-twisted sheaf
  $V$ is \emph {stable} if for all proper locally split subsheaves $W \subset V$
  we have
$$
\mu (W) < \mu (V).
$$
The sheaf is \emph{semistable} of for all proper locally split
subsheaves $W\subset V$ we have
$$
\mu(W)\leq\mu(V).
$$
\end{defn}
By ``locally split'' we mean that there is a faithfully flat map
$Z\to\ms X$ such that there is a retraction of the inclusion
$W_Z\subset V_Z$.  (I.e., $W$ is locally a direct summand of $V$.)  As
in the classical case, the stack of stable sheaves is a $\G_m$-gerbe
over an algebraic space. If in addition we fix a determinant, then the
resulting stack is a $\m _ r$-gerbe over an algebraic space, where $r$
is the rank of the sheaves in question.

\begin{notn}
  Given an invertible sheaf $L$ on $X$, the stack of stable (resp.\
  semistable) $\ms{X}$-twisted sheaves of rank $n$ and determinant $L$
  will be denoted $\ms{M} _ {\ms{X} / k} (n, L)$ (resp.\ $\ms M_{\ms
    X/k}(n,L)^{ss}$). The coarse moduli space (which in the stable case is
  also the sheafification) will be denoted $M _ {\ms{X} / k} (n, L)$
  (resp.\ $M_{\ms X/k}(n,L)^{ss}$).

  Given an integer $d$, the stack of stable 
  $\ms{X}$-twisted sheaves of rank $n$ and degree $d$ will be denoted
  $\ms{M} _ {\ms{X} / k} (n, d)$, and its coarse moduli space will be
  denoted $M _ {\ms{X} / k} (n, d)$.
\end{notn}

As mentioned above, $\ms{M} _ {\ms{X} / k} (n, d) \to M _ {\ms{X} / k}
(n, d)$ is a $\G_m$-gerbe; similarly, $\ms{M} _ {\ms{X} / k} (n, L)
\to M _ {\ms{X} / k} (n, L)$ is a $\m _ n$-gerbe. In fact, the stack
$\ms{M} _ {\ms{X} / k} (n, L)$ is the stack theoretic fiber of the
determinant morphism $\ms{M} _ {\ms{X} / k} (n, d) \to \sPic _ {X / k}
^ d$ over the morphism $\spec k \to \sPic ^ d _ {X / k}$ corresponding
to $L$. The comparison results described in Section \ref{Torelli
  section}, combined with classical results on the stack of stable
sheaves on a curve, show that $\ms{M} _ {\ms{X} / k} (n, d)$ is of
finite type over $k$; if we assume that the class $\ms{X}$ is zero in
$\H ^ 2 (X \tensor \widebar {k}, \m _ n)$ we know that $\ms{M} _
{\ms{X} / k} (n, d)$ is quasi-proper (i.e., satisfies existence part
of the valuative criterion of properness) whenever $n$ and $d$ are
relatively prime.

\subsection {Deformation theory}\label{defmn}

As $\ms O$-modules on a ringed topos (the \'etale topos of $\ms X$),
$\ms X$-twisted sheaves are susceptible to the usual deformation
theory of Illusie.  The following theorem summarizes the consequences
of this fact.  Suppose $S=\spec A$ is affine, $I\to\widetilde A\to A$
is a small extension (so that the kernel $I$ is an $A$-module),
$X/\widetilde A$ is flat, $\ms X\to X$ is a $\m_n$-gerbe with $n$
invertible in $A$, and $\ms F$ is an $A$-flat family of quasi-coherent
$\ms X\tensor_{\widetilde A}A$-twisted sheaves of finite presentation.

\begin{thm} \label{deformation theory} There is an element $\mf
  o\subset\ext^2(\ms F,\ms F_A\tensor I)$ such that
\begin{enumerate}
\item $\mf o=0$ if and only if there is an $\widetilde A$-flat
  quasi-coherent $\ms X$-twisted sheaf $\widetilde{\ms F}$ and an
  isomorphism $\widetilde{\ms F}\tensor_{\widetilde A}A\simto\ms F$;
\item if such an extension exists, the set of such extensions is a
  torsor under $\ext^1(\ms F,\ms F\tensor I)$;
\item given one such extension, the group of automorphisms of
  $\widetilde F$ which reduce to the identity on $\ms F$ is identified
  with $\Hom(\ms F,\ms F\tensor I)$.
\end{enumerate}
\end{thm}

\begin{cor}
  If $X/S$ is a relative curve, then the stack of locally free $\ms
  X$-twisted sheaves is smooth over $S$.
\end{cor}

Deformation theory can also be used to construct global objects from
local data.  The key technical tool is the following; see
\cite{lieblich_period_2007} for a more detailed proof and further references.

\begin{prop} \label{lift starter} Let $\ms P$ be a separated tame
  Artin stack of finite type over $k$ which is pure of dimension $1$.
  Let $P$ be the coarse moduli space of $\ms P$ and let $\ms C\to\ms
  P$ be a $\m_n$-gerbe.  Suppose $\ms P$ is regular away from the
  (finitely many) closed residual gerbes $\xi_1,\ldots,\xi_r$.
  Suppose the map $\Pic(\ms P)\to\prod\Pic(\xi_i)$ has kernel
  generated by the image of $\Pic(P)$ under pullback.
\begin{enumerate}
\item Given locally free $\ms C_{\xi_i}$-twisted sheaves $V_i$ of rank
  $m$, $i=1,\dots,r$, and locally free $\ms C_\eta$-twisted sheaf
  $V_\eta$ of rank $m$, there is a locally free $\ms C$-twisted sheaf
  $W$ of rank $m$ such that $W_\eta\cong V_\eta$ and $W_{\xi_i}\cong
  V_i$.
\item Suppose that $k$ is finite, $P\cong\P^1$, $\ms P\to P$ is
  generically a $\m_n$-gerbe, and the ramification extension $R\to P$
  of $\ms C$ (see Proposition \ref{intrinsic ramification} below) is
  geometrically connected.  Given an invertible sheaf $L\in\Pic(\ms
  P)$, if there are $\ms C_{\xi_i}$-twisted sheaves $V_i$ of rank $m$
  such that $\det V_i\cong L_{\xi_i}$ and a $\ms C_\eta$-twisted sheaf
  $V_\eta$ of rank $m$, then there is a locally free $\ms C$-twisted
  sheaf $V$ of rank $m$ such that $\det V\cong L$.
\end{enumerate}
\end{prop}
\begin{proof}
  Let $P$ be the coarse moduli space of $\ms P$, and let $p_i\in P$ be
  the image of $\xi_i$.  Let $\widehat{\eta}_i$ denote the
  localization of $\spec\widehat{\ms O}_{p_i}$ at the set of maximal
  (i.e., generic) points.  Let $\ms O(1)$ be an ample invertible sheaf on
  $P$.

  To prove the first item, note that the stack $\ms C_{\xi_i}$ is
  tame, as its inertia group is an extension of a reductive group by a
  $\m_n$.  Thus, the infinitesimal deformations of each $V_i$ are
  unobstructed.  Since $\ms C$ is proper over $P$, the Grothendieck
  existence theorem implies that for each $i$ there is a $\ms
  C\times\spec\widehat{\ms O}_{P,p_i}$-twisted sheaf $\widetilde{V}_i$
  of rank $m$ whose restriction to $\ms C_{\xi_i}$ is $V_i$.  On the
  other hand, since the scheme of generic points of $P$ is
  $0$-dimensional, we know that for each $i$ there is an isomorphism
  $(\widetilde{V}_i)_{\widehat{\eta}_i}\cong
  (V_\eta)_{\widehat{\eta}_i}$.  (See \cite{lieblich_period_2007} for
  a similar situation, with more details.)  The basic descent result
  of \cite{MR1432058} shows that we can glue the $\widetilde{V}_i$ to
  $V_\eta$ to produce $W$, as desired.

  To prove the second item, choose any $W$ as in the first part, and
  let $L'=\det W$.  This is an invertible sheaf which is isomorphic to
  $L$ in a neighborhood of each $\xi_i$, and therefore (by hypothesis)
  there is an invertible sheaf $M$ on $P$ such that $L\tensor
  (L')^{-1}\cong M_{\ms P}$.  Twisting $W$ by a suitable (negative!)
  power of $\ms O(1)$, we may assume that $M$ is ample of arbitrarily
  large degree. By the Lang-Weil estimates and hypothesis that $P
  \cong \P ^ 1$, there is a point $q \in R$ whose image $p$ in $P$ is
  an element of $| M |$. Making an elementary transformation of $W$
  along $p$ (over which the Brauer class associated to $\ms{C}$ is
  trivial) produces a locally free $\ms{C}$-twisted sheaf with the
  desired properties.
\end{proof}

\section {Basic facts on the Brauer group}

In this appendix, we review a few basic facts about the Brauer group
of a scheme. We freely use the technology of twisted sheaves, as
introduced in the previous appendix. Let $Z$ be a quasi-compact
separated scheme and $\ms{Z} \to Z$ a $\G_m$-gerbe. We start with a
(somewhat idiosyncratic) definition of the Brauer group of $Z$.

\begin{defn} \label{Brauer group} The cohomology class $[\ms{Z}] \in
  \H ^ 2 (Z, \G_m)$ is said to \emph {belong to the Brauer group of
    $Z$} if there is a non-zero locally free $\ms{Z}$-twisted sheaf of
  finite rank.
\end{defn}

\begin{lem} \label{the Brauer group is a group}
The Brauer group of $Z$ is a group.
\end{lem}
\begin{proof}
  Given two $\G_m$-gerbes $\ms{Z} _ 1 \to Z$ and $\ms{Z} _ 2 \to Z$
  which belong to the Brauer group of $Z$, let $V _ i$ be a locally
  free $\ms{Z} _ i$-twisted sheaf. Then $V _ 1 \tensor V _ 2$ is a
  locally free $\ms{Z} _ 1 \wedge \ms{Z} _ 2$-twisted sheaf, where
  $\ms{Z} _ 1 \wedge \ms{Z} _ 2$ is the $\G_m$-gerbe considered in
  \cite{MR0344253,krashen_index_2008}, which represents the cohomology
  class $[\ms{Z} _ 1] + [\ms{Z} _ 2]$. The neutral element in the
  group is represented by the trivial gerbe $\B \G_m \times Z \to Z$.
\end{proof}

We thus find a distinguished subgroup $\Br (Z) \subset \H ^ 2 (Z,
\G_m)$ containing those classes which belong to the Brauer group of
$Z$. What are the properties of this group?

\begin{prop} \label{basic properties of the Brauer group} The group
  $\Br (Z)$ has the following properties.
\begin{enumerate}
\item An element $[\ms{Z}]$ is trivial in $\Br (Z)$ if and only if
  there is an invertible $\ms{Z}$-twisted sheaf.
\item $\Br (Z)$ is a torsion abelian group.
\item (Gabber) If $Z$ admits an ample invertible sheaf, then the
  inclusion $\Br (Z) \subset \H ^ 2 (Z, \G_m) _ {\text{tors}}$ is an
  isomorphism.
\end{enumerate}
\end{prop}
\begin{proof} [Sketch of proof]
  If $\ms{L}$ is an invertible $\ms{Z}$-twisted sheaf, then we can
  define a morphism of $\G_m$-gerbes $\B \G_m \to \ms{Z}$ by sending
  an invertible sheaf $L$ to the invertible $\ms{Z}$-twisted sheaf $L
  \tensor \ms{L}$. Since any $\G_m$-gerbe admitting a morphism of
  $\G_m$-gerbes from $\B \G_m$ is trivial, we see that $\ms{Z}$ is
  trivial \cite{MR0344253}.

  Now suppose that $\ms{Z}$ is an arbitrary element of $\Br (Z)$, and
  let $V$ be a locally free $\ms{Z}$-twisted sheaf of rank
  $r$. Writing $\ms{Z} _ r$ for the gerbe corresponding to $r
  [\ms{Z}]$, one can show that the sheaf $\det V$ is an invertible
  $\ms{Z} _ r$-twisted sheaf, thus showing that $r [\ms{Z}] = 0$, as
  desired.

  The last part of the proposition is due to Gabber; a different proof has been
  written down by de Jong \cite{de_jong_a._j._result_???}.
\end{proof}

Given a $\G_m$-gerbe $\ms{Z} \to Z$ and a locally free
$\ms{Z}$-twisted sheaf $V$, the $\ms{O} _ {\ms{Z}}$-algebra $\send
(V)$ is acted upon trivially by the inertia stack of $\ms{Z}$, and
thus is the pullback of a unique sheaf of algebras $\ms{A}$ on
$Z$. The algebras $\ms{A}$ which arise in this manner are precisely
the Azumaya algebras: \'etale forms of $\M _ r (\ms{O} _ Z)$ (for
positive integers $r$). Moreover, starting with an Azumaya algebra, we
can produce a $\G_m$-gerbe by solving the moduli problem of
trivializing the algebra (i.e., making it isomorphic to a matrix
algebra). Further details about this correspondence may be found in
\cite[\S V.4]{MR0344253}.

When $Z = \spec K$ for some field $K$, an Azumaya algebra is precisely
a central simple algebra over $K$, and thus we recover the classical
Brauer group of the field. Note that in this case if $\ms{Z} \to Z$ is
a $\G_m$-gerbe, any nonzero coherent $\ms{Z}$-twisted sheaf is locally
free. Moreover, we know that $\ms{Z} \tensor \widebar {K}$ is
isomorphic to $\B \G_m \tensor \widebar {K}$, and therefore that there
is an invertible $\ms{Z} \tensor \widebar {K}$-twisted sheaf. Pushing
forward to $\ms{Z}$, we see that there is a nontrivial quasi-coherent
$\ms{Z}$-twisted sheaf $\ms{Q}$. Since $\ms{Z}$ is Noetherian,
$\ms{Q}$ is a colimit of coherent $\ms{Z}$-twisted subsheaves. This
shows that $\ms{Z}$ belongs to the Brauer group of $Z$. We have just
given a geometric proof of the classical Galois cohomological result
$\Br (K) = \H ^ 2 (\spec K, \G_m)$.

Now assume that $Z$ is integral and Noetherian with generic point
$\eta$, and let $\ms{Z} \to Z$ in arbitrary $\G_m$-gerbe. By the
preceding, there is a coherent $\ms{Z}$-twisted sheaf of positive
rank.

\begin{defn} \label{period and index} The \emph {index} of $[\ms{Z}]$,
  denoted $\ind ([\ms{Z}])$, is the minimal nonzero rank of a coherent
  $\ms{Z}$-twisted sheaf. The \emph {period} of $[\ms{Z}]$, denoted
  $\per ([\ms{Z}])$, is the order of $[\ms{Z}] _ {\eta}$ in $\H ^ 2
  (\eta, \G_m)$.
\end{defn}

We have written Definition \ref{period and index} so that it only
pertains to generic properties of $\ms{Z}$ and $Z$. One can imagine
more general definitions (e.g., using locally free $\ms Z$-twisted
sheaves; when $Z$ is regular of dimension 1 or 2, the basic properties
of reflexive sheaves tell us that the natural global definition will
actually equal the generic definition. When $Z = \spec K$, is easy to
say what the index of a Brauer class $\alpha$ is: $\alpha$
parameterizes a unique central division algebra over $K$, whose
dimension over $K$ is $n ^ 2$ for some positive integer $n$. The index
of $\alpha$ is then $n$.

The basic fact governing the period and index is the following.

\begin{prop} \label{basic period-index facts} For any $\alpha \in \Br
  (K)$, there is a positive integer $h$ such that $\per (\alpha) \mid
  \ind (\alpha)$ and $\ind (\alpha) \mid \per (\alpha) ^ h$.
\end{prop}

The proof of Proposition \ref{basic period-index facts} is an exercise
in Galois cohomology; the reader is referred to (for example)
\cite[Lemma 2.1.1.3]{lieblich_twisted_2008}. One immediate consequence
of Proposition \ref{basic period-index facts} is the question: how can
we understand $h$? For example, does $h$ depend only on $K$ (in the
sense that there is a value of $h$ which works for all $\alpha \in \Br
(K)$)? If so, what properties of $K$ are being measured by $h$? And so
on. Much work has gone into this problem; for a summary of our current
expectation the reader is referred to \cite{lieblich_twisted_2008}.

\section {Ramification of Brauer classes}

In this section, we recall the basic facts about the ramification
theory of Brauer classes. We also describe how to split ramification
by a stack. To begin with, we consider the case $X = \spec R$ with $R$
a complete discrete valuation ring with valuation $v$. Fix a uniformizer $t$ of
$R$; let $K$ and $\kappa$ denote the fraction field and residue field
of $R$, respectively. Write $j: \spec K \to \spec R$ for the natural
inclusion. Throughout, we only consider Brauer classes $\alpha \in \Br
(K)$ of period relatively prime to $\ch (\kappa)$; we write $\Br (K)
'$ for the subgroup of classes satisfying this condition.

The usual theory of divisors yields an exact sequence of \'etale
sheaves on $\spec R$
$$
0 \to \G_m \to j _ {*} \G_m \to \Z _ {(t)} \to 0,
$$
which yields a map $$\H ^ 2 (\spec R, j _ {*} \G_m) \to \H ^ 2 (\spec
R, \Z _ {(t)}) = \H ^ 2 (\kappa, \Z) = \H ^ 1 (\kappa, \Q / \Z).$$
Since any $a \in \Br (K) '$ has an unramified splitting field, the
Leray spectral sequence for $\G_m$ on $j$ shows that $\H ^ 2 (\spec R,
j _ {*} \G_m) ' = \Br (K) '$ (where the $'$ denotes classes with
orders invertible in $R$). Putting this together yields the \emph
{ramification sequence}
$$
0 \to \Br (R)' \to \Br (K)' \to \H ^ 1 (\kappa, \Q / \Z).
$$
The last group in the sequence parameterizes cyclic extensions of the
residue field $\kappa$. Suppose for the sake of simplicity that
$\kappa$ contains a primitive $n$th root of unity. The ramification of
a cyclic algebra $(a, b)$ is given by the extension of $\kappa$
generated by the $n$th root of $(-1) ^ {v (a) v (b)} a ^ {v (b)} / b ^
{v (a)}$. In particular, given any element $\widebar u \in \kappa ^
{*}$, the algebra $(u, t)$ has ramification extension $\kappa
(\widebar u ^ {1 / n})$, where $u$ is any lift of $\widebar u$ in $R ^
{*}$. With our assumption about roots of unity, any cyclic extension
of degree $n$ is given by extracting roots of some $\widebar u$. This
has the following two useful consequences. Given a positive integer
$n$, let $\ms{R} _ n \to \spec R$ denote the stack of $n$th roots of
the closed point of $\spec R$, as in \cite{cadman_using_2007}.

\begin{prop} \label{ramification} Assume that $\kappa$ contains a
  primitive $n$th root of unity $\zeta$. Fix an element $\alpha \in
  \Br (K) [n]$.
\begin{enumerate}
\item There exists $u \in R ^ {*}$ and $\alpha ' \in \Br (R)$ such
  that $\alpha = \alpha ' + (u, t) \in \Br (K)$.
\item There exists $\beta \in \Br (\ms{R} _ n)$ whose image in $\Br
  (K)$ under the restriction map is $\alpha$.
\end{enumerate}
\end{prop}
\begin{proof}
  The first item follows immediately from the paragraph preceding this
  proposition: we can find $u$ such that $(u, t)$ has the same
  ramification as $\alpha$, and subtracting this class yields an
  element of $\Br (R)$. To prove the second item, it follows in the
  first that it suffices to prove it for the class $(u, t)$. Recall
  that $\ms{R} _ n$ is the stacky quotient of $\spec R [t ^ {1 / n}]$
  by the natural action of $\m _ n$; to extend $(u, t)$ to $\ms{R} _
  n$, it suffices to find a $\m _ n$-equivariant Azumaya algebra in
  $(u, t) _ {K (t ^ {1 / n})}$. Recall that $(u, t)$ is generated by
  $x$ and $y$ such that $x ^ n = u$, $y ^ n = t$, and $xy = \zeta
  yx$. Letting $\tilde y = y / t ^ {1 / n}$, the natural action of $\m
  _ n$ on $t ^ {1 / n}$ yields an equivariant Azumaya algebra in $(u,
  t) _ {K (t ^ {1 / u})}$, as desired.
\end{proof}

The following corollary is an example of how one applies Proposition
\ref{ramification} in a global setting. More general results are true,
using various purity theorems.

\begin{cor} \label{global ramification} Let $C$ be a proper regular
  curve over a field $k$ which contains a primitive $n$th root of
  unity for some $n$ invertible in $k$. Suppose $\alpha \in \Br (k
  (C)) [n]$ is ramified at $p _ 1, \dots, p _ r$, and let $\ms{C} \to
  C$ is a stack of $n$th roots of $p _ 1+ \dots + p _ r$. There is an
  element $\alpha ' \in \Br (\ms{C}) [n]$ whose restriction to the
  generic point is $\alpha$.
\end{cor}
\begin{proof}
  Let $A$ be a central simple algebra over $k (C)$ with Brauer class
  $\alpha$. For any point $q \in C$, we know (at the very least) that
  there is a positive integer $m$ and an Azumaya algebra over the
  localization $\ms{C} _ q$ which is contained in $M _ m (A)$. Suppose
  $U \subset \ms{C}$ is an open substack over which there is an
  Azumaya algebra $B$ with generic Brauer class $\alpha$. Given a
  closed point $q \in C \setminus U$, we can choose some $m$ such that
  there is an Azumaya algebra $B '$ over $\widehat{\ms{C}} _ q$ whose
  restriction to the generic point $\widehat{\eta}$ of $\widehat C_q$
  is isomorphic to $M _ m (B)_{\widehat{\eta}}$. Since $U \times_C
  \widehat{\ms{C}} _ q = \widehat{\eta}$, we can glue $B '$ to $M _ m
  (B)$ as in the proof of Proposition \ref{lift starter} (using
  \cite{MR1432058}) to produce an Azumaya algebra over $U \cup \{ q
  \}$. Since $C \setminus U$ is finite, we can find some $m$ such that
  $M _ m (A)$ extends to an Azumaya algebra over $\ms{C}$, as desired.
\end{proof}

The following is a more complicated corollary of Proposition \ref{ramification}, using purity of the Brauer group on a surface.  We omit the proof.

\begin{cor} \label{surface splitting} Let $X$ be a connected regular
  Noetherian scheme pure of dimension $2$ with function field $K$.
  Suppose $U\subset X$ is the complement of a simple normal crossings
  divisor $D=D_1+\dots+D_r$ and $\alpha\in\Br(U)[n]$ is a Brauer class
  such that $n$ is invertible in $\kappa(D_i)$ for $i=1,\dots,r$.  If
  $\ms X\to X$ is the root construction of order $n$ over each $D_i$
  then there is a class $\widetilde{\alpha}\in\Br(\ms X)$ such that
  $\widetilde{\alpha}_U=\alpha$.
\end{cor}

We end this discussion with an intrinsic characterization of the
ramification extension as a moduli space.

\begin{prop} \label{intrinsic ramification} Let $X$ be a scheme on
  which $n$ is invertible such that $\Gamma (X, \ms{O})$ contains a
  primitive $n$th root of unity. Suppose $\pi: \ms{X} \to X$ is a $\m
  _ n$-gerbe. Given a further $\m _ n$-gerbe $\ms{Y} \to \ms{X}$, the
  relative Picard stack ${}_\tau\sPic _ {\ms{Y} / X}$ of invertible
  $\ms Y$-twisted sheaves is a $\G_m$-gerbe over a $\Z / n \Z$-torsor
  $T$. Moreover, the Brauer class of $\ms{Y}$ in $\H ^ 2 (\ms{X},
  \G_m)$ is the pullback of a class from $X$ if and only if $T$ is
  trivial.
\end{prop}
\begin{proof}
  Standard methods show that ${}_\tau\sPic _ {\ms{Y} / X}$ is a
  $\G_m$-gerbe over a flat algebraic space of finite presentation $P
  \to X$. Since the relative Picard space $\Pic _ {\ms{X} / X}$ is
  isomorphic to the constant group scheme $\Z / n \Z$, tensoring with
  invertible sheaves on $\ms{X}$ gives an action
$$
\Z / n \Z \times P \to P.
$$
To check that this makes $B$ a torsor, it suffices (by the obvious
functoriality of the construction) to treat the case in which $X$ is
the spectrum of an algebraically closed field $k$. In this case, $\Br
(\ms{X}) = 0$; choosing an invertible $\ms{Y}$-twisted sheaf $L$, we
see that all invertible $\ms{Y}$-twisted sheaves $M$ of the form $L
\tensor \Lambda$, where $\Lambda$ is invertible sheaf on
$\ms{X}$. This gives the desired result.
\end{proof}

Starting with a complete discrete valuation ring $R$ with fraction of $K$, an
element $\alpha \in \Br (K)$ gives rise to two cyclic extensions of
the residue field $\kappa$ of $R$: the classical ramification
extension and the moduli space produced in Proposition \ref{intrinsic
  ramification}. In fact, these two extensions are isomorphic. We will
use this fact, but we omit the details.

\end{document}